\documentclass[12pt]{amsart}
\usepackage{amsfonts}
\usepackage{bbm}
\usepackage{esint}
\usepackage{mathrsfs}
\usepackage{amssymb,amsmath,amsfonts,amsthm}
\usepackage[colorlinks, citecolor=blue,  hypertexnames=false]{hyperref}

\numberwithin{equation}{section}

\allowdisplaybreaks

\begin{document}

\baselineskip 16.1pt \hfuzz=6pt

\theoremstyle{plain}
\newtheorem{theorem}{Theorem}[section]
\newtheorem{prop}[theorem]{Proposition}
\newtheorem{lemma}[theorem]{Lemma}
\newtheorem{corollary}[theorem]{Corollary}
\newtheorem{example}[theorem]{Example}
\newtheorem*{thmA}{Theorem 1}
\newtheorem*{thmB}{Theorem 2}
\newtheorem*{thmC}{Theorem C}
\newtheorem*{thmD}{Theorem D}
\newtheorem*{thmE}{Theorem E}
\newtheorem*{thmF}{Theorem F}
\newtheorem*{thmG}{Theorem G}
\newtheorem*{thmH}{Theorem H}

\newtheorem*{defnA}{Definition A}
\newtheorem*{defnB}{Definition B}
\newtheorem*{defnC}{Definition C}
\newtheorem*{defnD}{Definition D}
\newtheorem*{defnE}{Definition E}

\theoremstyle{definition}
\newtheorem{definition}[theorem]{Definition}
\newtheorem{remark}[theorem]{Remark}

\renewcommand{\theequation}
{\thesection.\arabic{equation}}

\allowdisplaybreaks

\newcommand{\XX}{X}
\newcommand{\GG}{\mathop G \limits^{    \circ}}
\newcommand{\GGs}{{\mathop G\limits^{\circ}}}
\newcommand{\GGtheta}{{\mathop G\limits^{\circ}}_{\theta}}
\newcommand{\xoneandxtwo}{X_1\times\mathcal X_2}

\newcommand{\GGp}{{\mathop G\limits^{\circ}}}

\newcommand{\GGpp}{{\mathop G\limits^{\circ}}_{\theta_1,\theta_2}}

\newcommand{\e}{\varepsilon}
\newcommand{\bmo}{{\rm BMO}}
\newcommand{\vmo}{{\rm VMO}}
\newcommand{\cmo}{{\rm CMO}}
\newcommand{\Z}{\mathbb{Z}}
\newcommand{\N}{\mathbb{N}}
\newcommand{\R}{\mathbb{R}}
\newcommand{\C}{\mathbb{C}}

\pagestyle{myheadings}\markboth{\rm\small Yanchang Han, Yongsheng Han and  Ji Li
}{\rm\small Embedding theorems space of homogeneous type}

\title[Embedding theorems ]
{Geometric characterizations of embedding theorems }

\author{Yanchang Han, Yongsheng Han and  Ji Li}

%\thanks{The third author is supported by the
%Australian Research Council under Grant No.~ARC-DP120100399}



\begin{abstract}
The embedding theorem arises in several problems from analysis and geometry. The purpose of this paper is to provide a deeper understanding of analysis and geometry with a particular focus on embedding theorems on spaces of homogeneous type in the sense of Coifman and Weiss. We prove that embedding theorems hold on spaces of homogeneous type if and only if geometric conditions, namely the measures of all balls have lower bounds, hold. As applications, our results provide new and sharp previous related embedding theorems for the Sobolev, Besov and Triebel-Lizorkin spaces.
\end{abstract}
\maketitle

\section{Introduction }\label{sec:introduction}
\setcounter{equation}{0}

The purpose of this paper is to provide a deeper understanding of analysis and geometry with a particular focus on embedding theorems on spaces of homogeneous type which were introduced by Coifman and Weiss in the early 1970s, in \cite{CW1}. The original motivation to introduce spaces of homogeneous type is to carry out the Caldr\'on-Zygmund theory on locally compact abelian groups to a general geometric framework. Spaces of homogeneous type in the sense of Coifman and Weiss have been become a standard setting for harmonic analysis related to maximal function, differentiation theorem, singular integrals, function spaces such as Hardy spaces and functions of bounded mean oscillation, and many others. As Meyer remarked in his
preface to~\cite{DH}, \emph{``One is amazed by the dramatic
changes that occurred in analysis during the twentieth century.
In the 1930s complex methods and Fourier series played a
seminal role. After many improvements, mostly achieved by the
Calder\'on--Zygmund school, the action takes place today on
spaces of homogeneous type. No group structure is available,
the Fourier transform is missing, but a version of harmonic
analysis is still present. Indeed the geometry is conducting
the analysis.''}

We say that $(X,d,\mu)$ is {\it a space of homogeneous type} in the
sense of Coifman and Weiss if $d$ is a quasi-metric on~$X$
and $\mu$ is a nonzero measure satisfying the doubling
condition. More precisely, a \emph{quasi-metric}~$d$ on a set~$X$ is a
function $d:X\times X\rightarrow[0,\infty)$ satisfying
(i) $d(x,y) = d(y,x) \geq 0$ for all $x,y\in X$; (ii)
$d(x,y) = 0$ if and only if $x = y$; and (iii) the
\emph{quasi-triangle inequality}: there is a constant $A_0\in
[1,\infty)$ such that for all $x$, $y$, $z\in X$,
\begin{eqnarray}\label{eqn:quasitriangleineq}
    d(x,y)
    \leq A_0 [d(x,z) + d(z,y)].
\end{eqnarray}
We define the quasi-metric ball by $B(x,r) := \{y\in X: d(x,y)
< r\}$ for $x\in X$ and $r > 0$.
We say that a nonzero measure $\mu$ satisfies the
\emph{doubling condition} if there is a constant $C_\mu$ such
that for all $x\in\XX$ and $r > 0$,
\begin{eqnarray}\label{doubling condition}
  \mu(B(x,2r))
   \leq C_\mu \mu(B(x,r))
   < \infty.
\end{eqnarray}

We point out that the doubling condition (\ref{doubling condition}) implies that there exist positive constants
$\omega$ (the \emph{upper dimension} of~$\mu$) and $C$ such
that for all $x\in X$, $\lambda\geq 1$ and $r > 0$,
\begin{eqnarray}\label{upper dimension}
    \mu(B(x, \lambda r))
    \leq C\lambda^{\omega} \mu(B(x,r)).
\end{eqnarray}

Note that there is no differentiation structure on spaces of homogeneous type and even the original quasi-metric~$d$ may have no any regularity and quasi-metric balls, even Borel sets, may not be open. By the end of the 1970s, it was well recognized that much contemporary real analysis requires little structure on the underlying space. For instance, to obtain the maximal function characterizations for the Hardy spaces on spaces of homogeneous type, Mac\'ias and Segovia proved in~\cite{MS1} that one can
replace the quasi-metric~$d$ by another quasi-metric $d'$
on~$X$ such that the topologies induced on~$\XX$ by $d$
and~$d'$ coincide, and $d'$ has the following regularity
property:
\begin{eqnarray}\label{smooth metric}
    |d'(x,y) - d'(x',y)|
    \le C_0 \, d'(x,x')^\theta \,
        [d'(x,y) + d'(x',y)]^{1 - \theta}
\end{eqnarray}
for some constant~$C_0,$ some regularity exponent
$\theta\in(0,1)$, and for all $x$, $x'$, $y\in X$. Moreover, if
quasi-metric balls are defined by this new quasi-metric $d'$,
that is, $B'(x,r) := \{y\in X: d'(x,y) < r\}$ for $r > 0$, then
the measure $\mu$ satisfies the following property:
\begin{eqnarray}\label{regular}
    \mu(B'(x,r))\sim r.
\end{eqnarray}
Note that property~\eqref{regular} is much stronger than the
doubling condition. Mac\'{i}as and Segovia established the
maximal function characterization for Hardy spaces $H^p(X)$
with $(1 + \theta)^{-1} < p \leq 1$, on spaces of homogeneous
type~$(X,d',\mu)$ that satisfy the regularity
condition~\eqref{smooth metric} on the quasi-metric~$d'$ and
property~\eqref{regular} on the measure~$\mu$; see~\cite{MS2}.

The seminal work on these spaces $(X, d', \mu)$ that satisfy the regularity
condition~\eqref{smooth metric} on the quasi metric~$d'$ and property~\eqref{regular} on the measure~$\mu$ is the $T(b)$ theorem of David--Journ\'{e}--Semmes~\cite{DJS}. The
crucial tool in the proof of the $T(b)$ theorem is the
existence of a suitable approximation to the identity provided by these geometric conditions given in ~\eqref{smooth metric} and ~\eqref{regular}. The construction of such an approximation to the identity is due to
Coifman. More precisely, $S_k(x,y),$ the kernel of the approximation to the identity ~$S_k$, satisfies the
following conditions: for some constants $C > 0$ and $\e > 0$,
\begin{eqnarray*}
    &\textup{(i)}& S_k(x,y) = 0 {\rm\ for\ } d'(x,y) \geq C2^{-k},
        {\rm\ and\ } \| S_k\|_{\infty} \leq C2^k,\\
    &\textup{(ii)}& |S_k(x,y)-S_k(x',y)|
        \leq C2^{k(1+\e)}d'(x,x')^{\e},\\
    &\textup{(iii)}& |S_k(x,y)-S_k(x,y')|
        \leq C2^{k(1+\e)}d'(y,y')^{\e}, \quad\text{and}\\
    &\textup{(iv)}& \int_{X}S_k(x,y) \, d\mu(y)
        = 1
        = \int_{X}S_k(x,y) \, d\mu(x).
\end{eqnarray*}

Let $D_k := S_{k+1} - S_k$. In \cite{DJS}, the
Littlewood--Paley theory for $L^p(X)$, $1 < p < \infty$, was
established; namely, if $\mu(X) = \infty$ and $\mu(B(x,r))>0$
for all $x\in X$ and $r>0,$ then for each $p$ with $1 < p <
\infty$ there exists a positive constant $C_p$ such that
\[
    C_p^{-1}\|f\|_p
    \leq \big\|\big\{\sum_{k}|D_k(f)|^2\big\}^{\frac{1}{2}}\big\|_p
    \leq C_p\| f\|_p.
\]
The above estimates were the key tool for proving the
$Tb$~theorem on $(X, d', \mu)$; see \cite{DJS} for more
detail. Later, almost all important results on Euclidean space such as the Calder\'on reproducing formula, test
function spaces and distributions, the Littlewood--Paley
theory, function spaces and the embedding theorems were carried over on $(X, d', \mu)$ that satisfy the regularity
condition~\eqref{smooth metric} on the quasi-metric~$d'$ and
property~\eqref{regular} on the measure~$\mu.$ See
\cite{H1,H2}, \cite{HS}, \cite{HL} and \cite{DH} for more
details.

To devote to the so-called first order calculus, systematic theories on metric measure spaces were developed since the end of the 1970s, see, for example, \cite{Ch, HK, He1, He2}. Note that metric measure spaces are spaces of homogeneous type in the sense of Coifman and Weiss. For instance, Ahlfors $\omega$-regular space $(X, d, \mu)$ is such a metric measure space with $r^\omega/C\leq \mu(B(x,r))\leq Cr^\omega$ for some $\omega>0$ and $C$ independent of $x$ and $0<r\leq \displaystyle\sup_{x, y \in X} d (x, y).$ Ahlfors regular spaces are closely related to many questions such as complex analysis, singular integrals, the estimates of the bounds for heat kernels. See \cite{A}, \cite{ARSW}, \cite{BCG} and \cite{MMV} for more details.

Another important metric measure space is the Carnot-Carath\'edory space. There are several equivalent definitions for the Carnot-Carath\'edory distance, see \cite{JS} and \cite{NSW}. Here we mention Nagel and Stein's work in \cite{NS} on the Carnot-Carath\'edory space. Let $M$ be a connected smooth manifold
and $\{\mathbb{X}_1, \cdots, \mathbb{X}_k\}$ are $k$ given smooth
real vector fields on $M$ satisfying H\"{o}rmander
condition of order $m$, that is, these vector fields together with
their commutators of order $\leq m$ span the tangent space to $M
$ at each point. The most important geometric objects
used on the Carnot-Carath\'edory spaces are (i) a class of equivalent control
distances constructed on $M$ via the vector fields
$\{\mathbb{X}_1, \cdots, \mathbb{X}_k\};$ (ii) the volumes of balls
satisfying the doubling property and the certain lower bound estimates. More precisely, $\mu(B(x, sr)) \sim
s^{m+2} \mu(B(x,r))$ for $s\geq 1$ and $\mu(B(x, sr)) \sim
s^4\mu(B(x,r))$ for $s\leq 1.$ These conditions on the measure
are weaker than property~\eqref{regular} but are still stronger
than the original doubling condition~\eqref{doubling
condition}. The Carnot-Carath\'edory spaces, as natural model geometries, are closely related to hypoelliptic partial differential equations, subelliptic operators, CR geometry and quasiconformal mapping. See \cite{FP}, \cite{F}, \cite{FL}, \cite{GN1, GN2}, \cite{K}, \cite{NSW}, \cite{SC} and \cite{VSC}.

In~\cite{HMY1}, motivated by the work of Nagel and
Stein, the Hardy spaces were developed on spaces of homogeneous
type with a regular quasi-metric and a measure satisfying the
above conditions. Moreover, in \cite{HMY2} singular integrals and the Besov and Triebel-Lizorkin spaces were also developed on spaces of homogeneous type $(X, d, \mu)$ where the quasi-metric $d$ satisfies the H\"older regularity
property~\eqref{smooth metric} and the measure $\mu$ satisfies
the doubling property together with the reverse doubling condition, that
is, there are constants $\kappa \in (0, \omega]$ and $c \in (0, 1]$ such
that
\begin{eqnarray}\label{reversed doubling condition}
c \lambda^\kappa \mu ( B (x, r) ) \leq \mu ( B (x, \lambda r) )
\end{eqnarray}
for all $x \in X$, $0 < r < \displaystyle\sup_{x, y \in X} d (x, y) / 2$ and $1
\leq \lambda < \displaystyle\sup_{x, y \in X} d (x, y) / 2r $.

Recently, in \cite{GKZ} and \cite{HLW}, by different approaches, the Besov and the Triebel-Lizorkin spaces on metric measure spaces $(X, d, \mu)$ with the measure $\mu$ satisfying doubling property only and the Hardy spaces on spaces of homogeneous type in the sense of Coifman and Weiss were developed, respectively.

However, whether the most important embedding theorems can be established on spaces of homogeneous type in the sense of Coifman and Weiss is still an open problem. Even this is open whenever $(X,d,\mu)$ is a metric measure space with the measure $\mu$ satisfying the doubling together with the reverse doubling conditions.

The goal of this paper is to answer these problems.  Throughout the rest of the paper, we will work on the space of homogeneous type $(X,d,\mu)$ in the sense of Coifman and Weiss with $\mu(\{x\})=0$ and $0<\mu(B(x,r))<\infty$ for all $x\in X$ and $r>0.$
To characterize the embedding theorem on spaces of homogeneous type $(X,d,\mu),$ the crucial geometric condition is  the following
\begin{definition}\label{density}
Suppose that $(X,d,\mu)$ is a space of homogeneous type in the sense of Coifman and Weiss with the upper dimension $\omega.$
The measure $\mu$ is said to have a lower bound if there is a constant $C$ such that
$$\mu(B(x,r))\geq C r^{\omega}$$
for each $x\in X$ and all $r>0.$ And $\mu$ has a locally lower bound if
$$\mu(B(x,r))\geq C r^{\omega}$$
for each $x\in X$ and all $0<r\leq 1.$
\end{definition}
We remark that the lower bound conditions on the measure were used in the classical cases. To see this, let $(M,g)$ be a complete non-compact Riemannian manifold of dimension $n$ having non-negative curvature. Let $\mu$ be the canonical Riemannian measure on $M$ and denote by $V(x,r)$ the volume of the ball of radius $r>0$ centered at $x\in M$, i.e., $V(x,r)=\mu(B(x,r))$.
Note also that any smooth $(n-1)$-submanifold (i.e., hypersurface of co-dimension 1) inherits a Riemannian measure which we will denote by $\mu_{n-1}$.
It is well known that
from the celebrated Bishop--Gromov comparison theorem (see, e.g., \cite{C}), we have $V(x,2r)\leq 2^n V(x,r)$.
In this setting, the measure with lower bound condition is equivalent to the Sobolev-type inequality and related to the isoperimetric inequality and Poincar\'e's inequality. See Theorem 3.1.1 and Theorem 3.1.2 in \cite{SC} and see also \cite{CKP}.
In \cite{HKT}, it was proved that if the Sobolev embedding theorem holds in $\Omega\subset R^n,$ in any of all the possible cases, then $\Omega$ satisfies the
measure density condition, i.e. there exists a constant $c>0$ such that for all $x\in \Omega$ and all $0<r\leq 1$
$$|B(x,r)\cap \Omega|\geq cr^n.$$
In Section 2, we will show the main result of this paper. Applications of this main result will be given in the last section.

\section{Embedding Theorem}
\setcounter{equation}{0}

Before stating our main result, we recall
the following remarkable orthonormal wavelet basis which was constructed recently by Auscher and Hyt\"onen in \cite{AH}.
\begin{thmA}[\cite{AH} Theorem 7.1, \cite{HLW} Theorem 2.9 and Corollary 2.10]\label{theorem AH orth basi}
    Let $(X,d,\mu)$ be a space of homogeneous type in the sense of Coifman and Weiss with
    quasi-triangle constant $A_0.$ There exists an orthonormal wavelet basis
    $\{\psi_\alpha^k\}$, $k\in\mathbb{Z}$, $x_\alpha^k\in
    \mathscr{Y}^k$, of $L^2(X)$, having exponential decay
    \begin{eqnarray}\label{exponential decay}
        |\psi_\alpha^k(x)|
        \leq {C\over \sqrt{\mu(B(x_\alpha^k,\delta^k))}}
            \exp\Big(-\nu\Big( {d(x^k_\alpha,x)\over\delta^k}\Big)^a\Big),
    \end{eqnarray}
    H\"older regularity
    \begin{eqnarray}\label{Holder-regularity}
        |\psi_\alpha^k(x)-\psi_\alpha^k(y)|
        \leq \frac{C}{\sqrt{\mu(B(x_\alpha^k,\delta^k))}}
            \Big( \frac{d(x,y)}{\delta^k}\Big)^\eta
            \exp\Big(-\nu\Big( \frac{d(x^k_\alpha,x)}{\delta^k}\Big)^a\Big)
    \end{eqnarray}
    for $d(x,y)\leq \delta^k$, and the cancellation property
    \begin{eqnarray}\label{cancellation}
        \int_X \psi_\alpha^k(x)\,d\mu(x) = 0,
        \qquad \text{for }k\in\mathbb{Z}\qquad \text{and} \alpha^k\in
    \mathscr{Y}^k.
    \end{eqnarray}
    Moreover,
    \begin{eqnarray}\label{eqn:AH_reproducing formula}
        f(x)
        = \sum_{k\in\mathbb{Z}}\sum_{\alpha \in \mathscr{Y}^k}
            \langle f,\psi_{\alpha}^k \rangle \psi_{\alpha}^k(x)
    \end{eqnarray}
    in the sense of $L^2(X)$, $G_0(\beta',\gamma')$ for each $\beta'\in(0,\beta)$ and
    $\gamma'\in(0,\gamma)$ and the space $(\GGs_0(\beta,\gamma))'$ of distributions.
\end{thmA}
Here $\delta$ is a fixed small parameter, say $\delta \leq
10^{-3} A_0^{-10}$, $a= (1+2\log_2A_0)^{-1},$ and $C < \infty$, $\nu > 0$ and
$\eta\in(0,1]$ are constants independent of $k$, $\alpha$, $x$
and~$x_\alpha^k$. See \cite{AH} and \cite{HLW} for more notations and details of the proofs.

We now introduce the sequence spaces on spaces of homogeneous type in the sense of Coifman and Weiss as follows.
\begin{definition}\label{B T}
Suppose that $\omega$
    is the upper dimension of~$(X,d,\mu)$. For $-\infty<s<\infty$ and $0<p, q\leq \infty,$ we say that a sequence $\{\lambda_\alpha^k\}_{\alpha\in \mathscr{Y}^k ,{k\in\mathbb{Z}}}$ belongs to $\dot{b}^{s,q}_p$ if
    $$\|\{\lambda_\alpha^k\}\|_{\dot{{b}}^{s, q}_{p}} :=\Big\{\sum_{k\in\mathbb{Z}}\delta^{-ksq}
                        \bigg[ \sum_{ \alpha\in \mathscr{Y}^k} \Big( \mu ( Q_\alpha^k )^{1/p-1/2} |\lambda_\alpha^k| \Big)^p \bigg]^{q/p}\Big\}^{1/q}<\infty$$
    and a sequence $\{\lambda_\alpha^k\}_{\alpha\in \mathscr{Y}^k ,{k\in\mathbb{Z}}}$ belongs to ${b}^{s,q}_p$ if
    $$\|\{\lambda_\alpha^k\}\|_{{{ b}}^{s, q}_{p}} :=\Big\{\sum_{k\in\mathbb{Z}^+}\delta^{-ksq}
                        \bigg[ \sum_{ \alpha\in \mathscr{Y}^k} \Big( \mu ( Q_\alpha^k )^{1/p-1/2} |\lambda_\alpha^k| \Big)^p \bigg]^{q/p}\Big\}^{1/q}<\infty,$$
where $Q_\alpha^k$ are dyadic cubes in $\mathscr{Y}^k.$

A sequence $\{\lambda_\alpha^k\}_{\alpha\in \mathscr{Y}^k, k\in\mathbb{Z}}$ belongs to $\dot{f}^{s,q}_p$ for $-\infty<s<\infty, 0<p<\infty, 0< q\leq \infty$ if
             $$\|\{\lambda_\alpha^k\}\|_{{\dot {f}}^{s, q}_{p}} :=\Big\|\Big\{\sum_{k\in\mathbb{Z}}
                    \sum \limits_{\alpha \in \mathscr{Y}^k}\delta^{-ksq}\Big(
        \mu ( Q_\alpha^k )^{-1/2}|\lambda_{\alpha}^k |
        \chi_{{Q}^{{k}}_{\alpha }}(x)\Big)^q\Big\}^{1/q}\Big\|_{L^p(X)}<\infty, $$
and a sequence $\{\lambda_\alpha^k\}_{\alpha\in \mathscr{Y}^k, k\in\mathbb{Z}^+}$ belongs to ${f}^{s,q}_p$ if
$$\|\{\lambda_\alpha^k\}\|_{{{f}}^{s, q}_{p}} :=\Big\|\Big\{\sum_{k\in\mathbb{Z}^+}
                    \sum \limits_{\alpha \in \mathscr{Y}^k}\delta^{-ksq}\Big(
        \mu ( Q_\alpha^k )^{-1/2}|\lambda_{\alpha}^k |
        \chi_{{Q}^{{k}}_{\alpha }}(x)\Big)^q\Big\}^{1/q}\Big\|_{L^p(X)}<\infty.$$
\end{definition}
We remark that the above sequence spaces on $R^n$ were introduced by Frazier and Jawerth in \cite{FJ} with $\delta =2^{-1}$ and  $\mu ( Q_\alpha^k )=2^{-kn}.$

The main result of this paper is the following embedding theorem.
\begin{theorem}\label{embedding}
(i) Let $0<p_i\leq \infty, 0<q\leq \infty,$ $i=1,2$ and $s_1\leq s_2$ with $ -\infty<s_1-\omega/p_1=s_2-\omega/p_2<\infty$. Then
$$\|\{\lambda_\alpha^k\}\|_{{\dot { b}}^{s_1,q}_{p_1}}\leq C \|\{\lambda_\alpha^k\}\|_{{\dot {b}}^{s_2,q}_{p_2}}$$
if and only if the measure $\mu$ has the lower bound and
$$\|\{\lambda_\alpha^k\}\|_{{{b}}^{s_1,q}_{p_1}}\leq C \|\{\lambda_\alpha^k\}\|_{{{b}}^{s_2,q}_{p_2}}$$
if and only if the measure $\mu$ has the locally lower bound.

(ii) Let $0<p_i< \infty$ and $0<q_i\leq \infty$ for $i=1,2,$
and $s_1\leq s_2$ with $ -\eta<s_1-\omega/p_1=s_2-\omega/p_2<\eta$. Then
$$\|\{\lambda_\alpha^k\}\|_{{\dot {f}}^{s_1,q_1}_{p_1}}\leq C \|\{\lambda_\alpha^k\}\|_{{\dot {f}}^{s_2,q_2}_{p_2}}$$
if and only if $\mu$ has the lower bound and
$$\|\{\lambda_\alpha^k\}\|_{{{f}}^{s_1,q_1}_{p_1}}\leq C \|\{\lambda_\alpha^k\}\|_{{{f}}^{s_2,q_2}_{p_2}}$$
if and only if $\mu$ has the locally lower bound.
\end{theorem}

To verify ``only if''-parts of Theorem \ref{embedding}, taking the sequence $\{\lambda_\alpha^k\}$ with $\lambda_{\alpha_0}^{k_0} =1$ and $\lambda_\alpha^k=0$ with $(\alpha, k)\not=(k_0,\alpha_0)$ for $k_0, k\in\mathbb{Z}$ and
                    $\alpha_0\in \mathscr{Y}^{k_0}, \alpha \in \mathscr{Y}^k,$ we then have
\begin{eqnarray*}
\|\{\lambda_{\alpha_0}^{k_0}\}\|_{\dot{{b}}^{s_1, q}_{p_1}(X)} &=& \delta^{-k_0s_1} \mu ( Q_{\alpha_0}^{k_0} )^{1/{p_1}-1/2}.
\end{eqnarray*}
Similarly,
\begin{eqnarray*}
\|\{\lambda_{\alpha_0}^{k_0}\}\|_{\dot{{b}}^{s_2, q}_{p_2}(X)} & =& \delta^{-k_0s_2} \mu ( Q_{\alpha_0}^{k_0} )^{1/{p_2}-1/2}.
\end{eqnarray*}
Therefore, if
\begin{eqnarray*}
\|\{\lambda_\alpha^k\}\|_{\dot{{b}}^{s_1, q}_{p_1}}\leq C \|\{\lambda_\alpha^k\}\|_{\dot{{b}}^{s_2, q}_{p_2}},
\end{eqnarray*}
we should have $\delta^{-k_0s_2} \mu ( Q_{\alpha_0}^{k_0} )^{1/{p_2}-1/2}\geq C \delta^{-k_0s_1} \mu ( Q_{\alpha_0}^{k_0} )^{1/{p_1}-1/2}$ and this implies that $\mu( Q_{\alpha_0}^{k_0} )\geq C\delta^{k_0\omega}$, for any $ k_0\in\mathbb{Z},\ \ \alpha_0\in\mathscr{Y}^{k_0}$. Repeating the same proof implies that if
\begin{eqnarray*}
\|\{\lambda_\alpha^k\}\|_{{b}^{s_1, q}_{p_1}}\leq C \|\{\lambda_\alpha^k\}\|_{ {b}^{s_2, q}_{p_2}},
\end{eqnarray*}
then $\mu( Q_{\alpha_0}^{k_0} )\geq C\delta^{k_0\omega}$, for any $ k_0\in\mathbb{Z}^+,\ \ \alpha_0\in\mathscr{Y}^{k_0}$.

The `` only if" parts for $\dot{f}^{s, q}_{p}(X)$ and ${f}^{s, q}_{p}(X)$ can be verified similarly to get $\mu( Q_{\alpha_0}^{k_0} )\geq C\delta^{k_0\omega}$ for any $ k_0\in\mathbb{Z},\ \ \alpha_0\in\mathscr{Y}^{k_0}$ and $\mu( Q_{\alpha_0}^{k_0} )\geq C\delta^{k_0\omega}$ for any $ k_0\in\mathbb{Z}^+,\ \ \alpha_0\in\mathscr{Y}^{k_0},$ respectively. Finally, the lower bound conditions will follow from the geometric structure of $(X, d,\mu),$ namely the following propositions.

\begin{prop}\label{prop key 1}
Suppose that for every $\ell\in\mathbb{Z}$ and every $\alpha\in \mathscr{Y}^\ell$,
\begin{eqnarray}\label{Yk}
\mu(Q_\alpha^\ell)\geq C\delta^{\ell\omega},
\end{eqnarray}
where $C$ is a positive
constant independent of $\ell$ and $\alpha$. Then we have
\begin{eqnarray}\label{Xk}
\mu(Q_\beta^k)\geq \widetilde{C}\delta^{k\omega}
\end{eqnarray}
for every $k\in\mathbb{Z}$ and every $\beta\in \mathscr{X}^k$, where $\widetilde{C}$ is a positive constant
independent of $k$ and $\beta$.
\end{prop}

Assume the above  proposition for the moment, we can obtain the following  result for the lower bound of the measure of
any balls in $X$, which provide the necessary condition for the embedding theorem.
\begin{prop}\label{prop key 2}
Suppose that for every $k\in\mathbb{Z}$ and every $\beta\in \mathscr{X}^k$,
\begin{eqnarray}\label{Xk 1}
\mu(Q_\beta^k)\geq C\delta^{k\omega},
\end{eqnarray}
where $C$ is a positive
constant independent of $k$ and $\alpha$. Then we have
\begin{eqnarray}\label{ball}
\mu(B(x,r))\geq \widetilde{C} r^\omega
\end{eqnarray}
for every $x\in X$ and every $r>0$, where $\widetilde{C}$ is a positive constant
independent of $x$ and $r$.
\end{prop}
Before proving Propositions \ref{prop key 1} and \ref{prop key 2}, we recall the fundamental result of the construction of dyadic cubes
by Hytonen and Kairema \cite{HK}.
\begin{thmB}[\cite{HK} Theorem 2.2]\label{theorem dyadic cubes}
    Suppose that constants $0 < c_0 \leq C_0 < \infty$ and
    $\delta\in(0,1)$ satisfy
    \begin{eqnarray}\label{test condition for cubes}
        12A_0^3C_0\delta
        \leq c_0.
    \end{eqnarray}
    Given a set of points $\{z_\alpha^k\}_{\alpha}$, $\alpha
    \in \mathscr{A}_k$, for every $k\in\mathbb{Z}$, with the
    properties that
    \begin{eqnarray}\label{sparse property}
        d(z_\alpha^k,z_\beta^k)
        \geq c_0\delta^k\
            (\alpha\not=\beta),\hskip1cm
        \min_\alpha d(x,z_\alpha^k)
        < C_0\delta^k,
            \qquad \text{for all $x\in X$},
    \end{eqnarray}
    we can construct families of sets $\widetilde{Q}_\alpha^k
    \subseteq Q_\alpha^k \subseteq
    \overline{Q}_\alpha^k$ (called open, half-open and closed
    \emph{dyadic cubes}), such that:
    \begin{eqnarray}
        && \widetilde{Q}_\alpha^k \mbox{ and } \overline{Q}_\alpha^k
            \mbox{ are the interior and closure of } Q_\alpha^k, \mbox{ respectively};\\
        &&\mbox{if } \ell\geq k, \mbox{ then either } Q_\beta^\ell\subseteq
            Q_\alpha^k \mbox{ or } Q_\alpha^k
            \cap Q_\beta^\ell=\emptyset;\label {DyadicP1}\\
        &&  X = \bigcup_\alpha Q_\alpha^k\ \ (\mbox{disjoint union})\qquad
            \text{for all $k\in\mathbb{Z}$};\label {DyadicP2} \\
        && B(z_\alpha^k,c_1\delta^k)\subseteq Q_\alpha^k\subseteq
            B(z_\alpha^k,C_1\delta^k),\ \mbox{where } c_1 := (3A_0^2)^{-1}c_0\
            \mbox{and}\  C_1 := 2A_0C_0;\label {prop_cube3}\\
        &&\mbox{if } \ell \geq k \mbox{ and } Q_\beta^\ell\subseteq Q_\alpha^k,
            \mbox{ then } B(z_\beta^\ell,C_1\delta^\ell)\subseteq
            B(z_\alpha^k,C_1\delta^k).\label {DyadicP4}
    \end{eqnarray}
    The open and closed cubes $\widetilde{Q}_\alpha^k$ and
    $\overline{Q}_\alpha^k$ depend only on the points
    $z_\beta^\ell$ for $\ell\geq k$. The half-open cubes
    $Q_\alpha^k$ depend on $z_\beta^\ell$ for $\ell\geq
    \min(k,k_0)$, where $k_0\in\mathbb{Z}$ is a preassigned
    number entering the construction.
\end{thmB}
\begin{proof}[Proof of Proposition \ref{prop key 1}]
%Now take $c_0 := 1$, $C_0 := 2A_0$ and $\delta \leq 10^{-3}
%A_0^{-10}$. By Theorem 2.2 in \cite{HK}, there
%exists a set of half-open dyadic cubes
%\[
%    \{Q_\alpha^k\}_{k\in\mathbb{Z},\alpha\in\mathscr{X}^k}
%\]
%associated with the reference dyadic points
%$\{x_\alpha^k\}_{k\in\mathbb{Z},\alpha\in\mathscr{X}^k}$.

For each fixed $k\in \mathbb{Z}$ and $\beta\in \mathscr{X}^k$, we now consider the number of the children of the dyadic cube $Q_\beta^k$.
Suppose now $Q_\beta^k$ has $M$ children and $M\geq2$, say $Q_{\beta_1}^{k+1},\ldots, Q_{\beta_M}^{k+1}$. Then from the construction of $\mathscr{Y}^k$ we get that
there must be one of the $M$ children belongs to $\mathscr{X}^{k+1}$ sharing the same center with $Q_\beta^k$, and the other $M-1$ children belong to $\mathscr{Y}^k$. Without lost of generality,
we assume that $Q_{\beta_1}^{k+1} \in \mathscr{X}^{k+1}$ and $ Q_{\beta_2}^{k+1},\ldots, Q_{\beta_M}^{k+1} \in \mathscr{Y}^{k}   $.
Then from \eqref{Yk} we get that $\mu( Q_{\beta_i}^{k+1} )\geq C\delta^{(k+1)\omega} $ for $i=2,\ldots,M$. As a consequence, we have
$$  \mu(Q_\beta^k)\geq \sum_{i=2}^M \mu( Q_{\beta_i}^{k+1} ) \geq C (M-1) \delta^{(k+1)\omega} = \widetilde{C} \delta^{k\omega},   $$
where we set $\widetilde{C}=C (M-1)\delta^\omega$.

Suppose now $Q_\beta^k$ has only one child, say $ Q_{\beta_1}^{k+1} $. Note that actually $ Q_{\beta_1}^{k+1} $ is the same as $ Q_{\beta}^{k} $. Suppose
$ Q_{\beta_1}^{k+1} $ has $M$ children in the $k+2$ level, $M\geq2$. Then using the previous argument we get that
$$  \mu(Q_\beta^k)\geq C (M-1) \delta^{(k+2)\omega} = \widetilde{C} \delta^{k\omega},   $$
where we set $\widetilde{C}=C (M-1)\delta^{2\omega}$.

If $ Q_{\beta_1}^{k+1} $ also has only one child, say $ Q_{\widetilde{\beta}_1}^{k+2}, $ then we further consider the number of the children of $ Q_{\widetilde{\beta}_1}^{k+2}.$

Thus, we now claim the following fact:

{\it
Suppose now we have a chain, $\{ Q_{\beta_0}^k$, $Q_{\beta_1}^{k+1}$, $Q_{\beta_2}^{k+2}, \ldots, $ $Q_{\beta_n}^{k+n}\}$, satisfying the condition that
$Q_{\beta_1}^{k+1}$ is the only one child of $Q_{\beta_0}^k$, $Q_{\beta_2}^{k+2}$ is the only one child of $Q_{\beta_1}^{k+1}$, and so on. Here $k\in\mathbb{Z}$
and $\beta_i\in \mathscr{X}^{k+i}$, $i=0,1,2,\ldots,n$.

Then there exists a positive constant $N$ such that for every possible chain satisfying the above conditions, we have $n \leq N$. In other words, $N$ is independent of $k$ and $\beta_0,\ldots,\beta_n$.
}

We now prove the above claim. Suppose $\{ Q_{\beta_0}^k$, $Q_{\beta_1}^{k+1}$, $Q_{\beta_2}^{k+2}, \ldots, $ $Q_{\beta_n}^{k+n}\}$ is an arbitrary chain satisfying the condition as in the claim.

First note that from (\ref{prop_cube3}) in Theorem \ref{theorem dyadic cubes}, we have
         $$
              B(y_{\beta_i}^{k+i},c_1\delta^{k+i})\subseteq Q_{\beta_i}^{k+i} \subseteq
              B(y_{\beta_i}^{k+i},C_1\delta^{k+i})
         $$
for $i=0,1,\ldots, n$, where  $c_1 := (3A_0^2)^{-1}c_0$  and  $C_1 := 2A_0C_0$.

Second, due to the condition of this chain, all the dyadic cubes $Q_{\beta_i}^{k+i}$ with $i=1,2,\ldots,n$ are the same as $Q_{\beta_0}^{k}$ and sharing the same center point. Thus,
we have
           $$
              B(y_{\beta_0}^{k},c_1\delta^{k})\subseteq Q_{\beta_0}^{k} = Q_{\beta_n}^{k+n} \subseteq
              B(y_{\beta_n}^{k+n},C_1\delta^{k+n}),
           $$
where  $y_{\beta_0}^{k}$ coincides with $y_{\beta_n}^{k+n}$. As a consequence, we get that
$$ c_1\delta^{k}\leq C_1\delta^{k+n}. $$

Note that from the setting of the space of homogeneous type $(X,d,\mu)$, we have the assumptions
that $ \mu(\{x\})=0 $ for every $x\in X$ and  that $0<\mu(B(x,r))<\infty$ for every $x\in X$ and $r>0$.
Thus, the balls $ \{ B(y_{\beta_i}^{k+i},c_1\delta^{k+i}) \} $ shrinks to the center point $y_{\beta_0}^{k}$ if $i$ tends to infinity, and
these balls must not be the same as $\{ y_{\beta_0}^{k}\}$ since there is no point mass.

This implies that
$$n\leq \log_{\delta^{-1}}\Big({C_1\over c_1}\Big), $$
which yields that the claim holds with $N=\big\lfloor\log_{\delta^{-1}}\Big({C_1\over c_1}\Big)\big\rfloor+1$.

Thus,for each fixed $k\in \mathbb{Z}$ and $\beta\in \mathscr{X}^k$, if the dyadic cube $Q_\beta^k$ has only one child $Q_{\beta_1}^{k+1}$, $Q_{\beta_1}^{k+1}$ has only one child $Q_{\beta_2}^{k+2}$ and so on, then this chain $\{Q_{\beta_i}^{k+i}\}$ must be finite and has at most $N$ cubes. That is,  $Q_{\beta_{N}}^{k+N}$ must have $M$ children and $M\geq 2$, say $Q_{\gamma_1}^{k+N+1},\ldots, Q_{\gamma_M}^{k+N+1}$, and there is only one of these $M$ children belonging to $\mathscr{X}^{k+N+1}$, all the other $M-1$ children
belonging to $\mathscr{Y}^{k+N}$.

Then using the previous argument again we get that
$$  \mu(Q_\beta^k)\geq C (M-1) \delta^{(k+N+1)\omega} = \widetilde{C} \delta^{k\omega},   $$
where we set $\widetilde{C}=C (M-1)\delta^{(N+1)\omega }$.
\end{proof}

\bigskip

\begin{proof}[Proof of Proposition \ref{prop key 2}]
Fix $x\in X$ and $r>0$. We set $ \alpha= ( 1+2\delta^{-1})^{-1} $. Next we choose $k\in\mathbb{Z}$ such that
$$ C_1\delta^{k+1}\leq \alpha r< C_1\delta^k, $$
where $C_1$ is the constant in Theorem \ref{theorem dyadic cubes}. Then we have
$$ (1-\alpha)r < r-C_1\delta^{k+1} ,$$
which implies that
$ C_1\delta^{k+1}<\alpha r $ and hence
\begin{eqnarray}\label{comparision}
 C_1\delta^{k} < {\alpha r  \over \delta} = {(1-\alpha)\over 2}r.
\end{eqnarray}

Now note that $B(x,\alpha r)$ must be cover by
a union of at most $M$ dyadic cubes in $\mathscr{X}^k$, since
1) $\cup_{\beta\in \mathscr{X}^k} Q_\beta^k = X$; 2) $\alpha r$ is comparable to the sidelength of the cubes in $\mathscr{X}^k$;
3) the doubling property of $\mu$ implies that the space $(X,d)$ is geometrically doubling.
Here we also point out that $M$ is a positive constant independent of $x$ and $r$.

Suppose now $Q_{\beta_1}^k,\ldots,Q_{\beta_n}^k$ is the dyadic cubes in $\mathscr{X}^k$ such that $\cup_{i=1}^n Q_{\beta_i}^n$ covers $B(x,\alpha r)$ and that
$Q_{\beta_i}^n\cap B(x,\alpha r)\not=\emptyset$, $n\leq M$.
Then  from (\ref{prop_cube3}) in Theorem \ref{theorem dyadic cubes}, we have
         $$
               Q_{\beta_i}^{k} \subseteq
              B(y_{\beta_i}^{k},C_1\delta^{k})
         $$
for $i=0,1,\ldots, n$.

We now point out that (\ref{comparision}) implies that
$ B(y_{\beta_i}^{k},C_1\delta^{k}) $ is contained in $B(x,r)$ for every
 $i=0,1,\ldots, n$. Hence, $Q_{\beta_i}^k$ is contained in $B(x,r)$ for every
 $i=0,1,\ldots, n$.
As a consequence, we have
\begin{eqnarray}
\mu(B(x,r))\geq \sum_{i=1}^n \mu(Q_{\beta_i}^k) \geq n \cdot C\delta^{k\omega}\geq n \cdot C \Big( {\alpha r\over C_1}\Big)^{\omega}\geq \widetilde{C}r^\omega,
\end{eqnarray}
where the second inequality follows from \eqref{Xk 1}, and $\widetilde{C}= C \big( {\alpha \over C_1}\big)^{\omega}$.
\end{proof}
\begin{corollary}\label{coro key 2}
Suppose that for every $k\in\mathbb{Z}$ and $k\geq0$, and for every $\alpha\in \mathscr{Y}^k$,
\begin{eqnarray}\label{Yk 1 coro}
\mu(Q_\alpha^k)\geq C\delta^{k\omega},
\end{eqnarray}
where $C$ is a positive
constant independent of $k$ and $\alpha$. Then we have
\begin{eqnarray}\label{ball coro}
\mu(B(x,r))\geq \widetilde{C} r^\omega
\end{eqnarray}
for every $x\in X$ and every $r>0$, where $\widetilde{C}$ is a positive constant
independent of $x$ and $r$.
\end{corollary}
\begin{proof}
We first claim that the condition \eqref{Yk 1 coro} for every $k\in\mathbb{Z}$ and $k\geq0$, and for every $\alpha\in \mathscr{Y}^k$ implies that
\begin{eqnarray}\label{Xk all k}
\mu(Q_\beta^\ell)\geq C\delta^{\ell\omega}
\end{eqnarray}
for every $\ell\in\mathbb{Z}$ and every $\beta\in \mathscr{X}^\ell$.

Suppose the above claim holds, then by applying the result in Proposition \ref{prop key 2}, we obtian that \eqref{ball coro} holds for every $x\in X$ and every $r>0$, where $\widetilde{C}$ is a positive constant
independent of $x$ and $r$.

Now we prove the claim. First suppose  that $\ell\geq0$ and $\beta\in \mathscr{X}^\ell$. Then following the same proof of Propsition \ref{prop key 1}, we obtain that
$$\mu(Q_\beta^\ell)\geq C\delta^{\ell\omega},$$
where $C$ is a positive constant
independent of $\ell$ and $\beta$.

Next consider $\ell=-1$ and $\beta\in \mathscr{X}^{-1}$. Note that all the decendent of the cube $Q_\beta^\ell$ are in $\mathscr{X}^L$ with the level index $L\geq0$.  Following the same proof of Propsition \ref{prop key 1} again, we obtain that
$$\mu(Q_\alpha^\ell)\geq C\delta^{\ell\omega},$$
where $C$ is a positive constant
independent of $\ell$ and $\beta$. Thus, the claim \eqref{Xk all k} holds for all $\ell\geq-1$.

By induction we obtain that the claim \eqref{Xk all k} holds for all $\ell<-1$ as well. This completes the proof of Corollary \ref{coro key 2}.
\end{proof}

We now show the `` if" parts of Theorem \ref{embedding}. Applying $\frac{p_2}{p_1}$-inequality for $\frac{p_2}{p_1}\leq 1$ yields
\begin{eqnarray*}\label{4.2}
\|\{b_\alpha^k\}\|_{\dot{ b}^{s_1, q}_{p_1}(X)} &=& \bigg\{\sum_{k\in \mathbb{Z}}\delta^{-ks_1q}\bigg[ \sum_{ \alpha\in \mathscr{Y}^k} \Big( \mu ( Q_\alpha^k )^{1/{p_1}-1/2} |b_\alpha^k| \Big)^{p_1} \bigg]^{q/{p_1}}
           \bigg\}^{1/q}\\
     &\leq & \bigg\{\sum_{k\in \mathbb{Z}}\delta^{-ks_1q}\bigg[ \sum_{ \alpha\in \mathscr{Y}^k} \Big(  \mu ( Q_\alpha^k )^{1/{p_1}-1/{p_2}}\mu ( Q_\alpha^k )^{1/{p_2}-1/2} |b_\alpha^k| \Big)^{p_2} \bigg]^{q/{p_2}}
           \bigg\}^{1/q}
     \\    &\leq &
         C\bigg\{\sum_{k\in \mathbb{Z}}\delta^{-ks_1q}\bigg[ \sum_{ \alpha\in \mathscr{Y}^k} \Big(  \delta^{k(\omega/{p_1}-\omega/{p_2})}\mu ( Q_\alpha^k )^{1/{p_2}-1/2} |b_\alpha^k| \Big)^{p_2} \bigg]^{q/{p_2}}
           \bigg\}^{1/q}
         \\& =& C \bigg\{\sum_{k\in \mathbb{Z}}\delta^{-ks_2q}\bigg[ \sum_{ \alpha\in \mathscr{Y}^k} \Big( \mu ( Q_\alpha^k )^{1/{p_2}-1/2} |b_\alpha^k| \Big)^{p_2} \bigg]^{q/{p_2}}
           \bigg\}^{1/q}\\ &=& C\|\{b_\alpha^k\}\|_{\dot{b}^{s_2, q}_{p_2}(X)},
\end{eqnarray*}
where the lower bound condition $\mu( Q_{\alpha_0}^{k_0} )\geq C\delta^{k_0\omega}$ together with the facts that ${{1}\over{p_1}}-{{1}\over{p_2}}\leq 0$ and $s_1-\omega/p_1=s_2-\omega/p_2$ are used in the last inequality and equality, respectively.

Repeating the same proof implies that if $\mu( Q_{\alpha}^{k} )\geq C\delta^{k\omega}$, for any $ k\in\mathbb{Z}^+,\ \ \alpha\in\mathscr{Y}^{k}$, then
\begin{eqnarray*}
\|\{b_\alpha^k\}\|_{{b}^{s_1, q}_{p_1}}\leq C \|\{b_\alpha^k\}\|_{ {b}^{s_2, q}_{p_2}}.
\end{eqnarray*}

To show the `` if" part of Theorem \ref{embedding} for sequence spaces in $ \dot{f}^{s, q}_{p}(X),$
by the homogeneity of the norm $\|\cdot\|_{\dot{f}^{s_2, q_2}_{p_2}(X)}$, we may suppose $\|\{\lambda_\alpha^k\}\|_{\dot{f}^{s_2, q_2}_{p_2}(X)}=1$ without loss of generality. Since
$$ \|\lambda_{\tau}^j\|_{{\dot{f}^{s_1, q_1}_{p_1}(X)}}^{p_1}
      = p_1\int_0^\infty t^{p_1-1}\mu\Big(\Big\{x:\Big\{\sum_{j=-\infty}^\infty\sum \limits_{\tau \in \mathscr{Y}^j}\delta^{-js_1q_1}\Big(
        \mu ( Q_\tau^j)^{-1/2}|\lambda_{\tau}^j |
        \chi_{{Q}^{{j}}_{\tau }}(x) \Big)^{q_1}
                        \Big\}^{1/q_1}>t\Big\}\Big)dt,$$
the point of departure of the proof is to estimate the following distribution function
$$
\mu\Big\{x:\ \Big\{\sum_{j=-\infty}^\infty\sum_{\tau \in \mathscr{Y}^j}\delta^{-js_1q_1}\Big(\mu ( Q_\tau^j)^{-1/2}|\lambda_{\tau}^j |\chi_{{Q}^{j}_{\tau }}(x) \Big)^{q_1}                        \Big\}^{1/q_1}>t\Big\}.
$$
To this end, note that by the orthogonality of wavelets $\psi_{\alpha}^k,$ we have
\begin{eqnarray}\label{4.3}
|\lambda_\tau^j| = \Big|\sum_{k\in\mathbb{Z}}\sum_{\alpha \in \mathscr{Y}^k}\lambda_\alpha^k
            \langle\psi_{\alpha}^k,\psi_{\tau}^j \rangle\Big|.
\end{eqnarray}

To estimate $\lambda_\tau^j,$ the key point is to replace the orthogonal estimate for $\langle\psi_{\alpha}^k,\psi_{\tau}^j \rangle$ by the following almost orthogonal estimate(see \cite{HLW} for the proof) for $j, k\in\mathbb{Z}, \alpha \in \mathscr{Y}^k, \tau \in \mathscr{Y}^j, \gamma>0$ and any $\epsilon: 0<\epsilon<\eta,$
\begin{eqnarray}\label{4.4}
&&\Big|\langle
            \psi_{\alpha}^k,\psi_{\tau}^j \rangle\Big|\\
            & \leq& C \delta^{| k - j | \epsilon}
\frac{\mu(Q_{\tau}^j)^{1\over2}
\mu(Q_{\alpha}^k)^{1\over 2}}{ V_{\delta^{(k \wedge j)}} (x_{\alpha}^k) +
V_{\delta^{(k \wedge j)}} (x_{\tau}^j) +V (x_{\alpha}^k,x_{\tau}^j)} \bigg(
\frac{\delta^{(k \wedge j)}}{ \delta^{(k \wedge j)} + d (x_{\alpha}^k,x_{\tau}^j) }
\bigg)^\gamma.\nonumber
\end{eqnarray}
We obtain that there exists a constant $C$ such that
\begin{eqnarray}\label{4.5}
&&   \sum \limits_{\tau \in \mathscr{Y}^j}
        \mu ( Q_\tau^j )^{-1/2}| \lambda_{\tau}^j |
        \chi_{{Q}^{{j}}_{\tau }}(x)  \nonumber
     \\    &\leq &C
         \sum_{k\in\mathbb{Z}} \sum \limits_{\tau \in \mathscr{Y}^j}
\sum_{\alpha \in \mathscr{Y}^k}|\lambda_{\alpha}^k |\chi_{{Q}^{{j}}_{\tau }}(x)
        \\&&\times\delta^{| k - j | \epsilon}
\frac{\mu(Q_{\alpha}^k)^{1\over2}}{V_{\delta^{(k \wedge j)}} (x_{\alpha}^k) + V_{\delta^{(k \wedge j)}} (x_{\tau}^j) +V (x_{\alpha}^k,x_{\tau}^j)} \bigg(
\frac{\delta^{(k \wedge j)}}{ \delta^{(k \wedge j)} + d (x_{\alpha}^k,x_{\tau}^j) }
\bigg)^\gamma.\nonumber
\end{eqnarray}

Now we claim that for $r\leq1$,
\begin{eqnarray}\label{4.6}
&&          \sum \limits_{\tau \in \mathscr{Y}^j}
\sum_{\alpha \in \mathscr{Y}^k}
        \frac{\mu(Q_{\alpha}^k)^{1\over2}}{V_{\delta^{(k \wedge j)}} (x_{\alpha}^k) + V_{\delta^{(k \wedge j)}} (x_{\tau}^j) +V (x_{\alpha}^k,x_{\tau}^j)} \bigg(
\frac{\delta^{(k \wedge j)}}{ \delta^{(k \wedge j)} + d (x_{\alpha}^k,x_{\tau}^j) }
\bigg)^\gamma|\lambda_{\alpha}^k |\chi_{{Q}^{{j}}_{\tau }}(x)
\\    &\leq &C\delta^{k\omega(1-\frac{1}{r})}\mu(B)^{\frac{1}{r}- 1}\inf\limits_{y\in B}\Big\{M\Big(\sum_{\alpha \in \mathbb{N}}\mu(Q_{\alpha}^k)^{-r/2}|\lambda_{\alpha}^k |^r\chi_{{Q}^{{k}}_{\alpha }}\Big)(y)\Big\}^\frac{1}{r},\nonumber
\end{eqnarray}
where $B=B(x,\delta^{k\wedge j})$ and $M$ is the Hardy--Littlewood maximal function.

To prove \eqref{4.6}, for $x\in \chi_{{Q}^{{j}}_{\tau }}$ first replacing $V_{\delta^{(k \wedge j)}} (x_{\alpha}^k) + V_{\delta^{(k \wedge j)}} (x_{\tau}^j) +V (x_{\alpha}^k,x_{\tau}^j)$ and $\delta^{(k \wedge j)} + d (x_{\alpha}^k,x_{\tau}^j)$ by $V_{\delta^{(k \wedge j)}} (x_{\alpha}^k) + V_{\delta^{(k \wedge j)}} (x) +V (x_{\alpha}^k,x)$ and $\delta^{(k \wedge j)} + d (x_{\alpha}^k,x),$ respectively, and then taking sum over $\tau \in \mathscr{Y}^j$ together with the fact $r\leq 1$ yield
\begin{eqnarray*}
&& \sum \limits_{\tau \in \mathscr{Y}^j}
\sum_{\alpha \in \mathscr{Y}^k}
        \frac{\mu(Q_{\alpha}^k)^{1\over2}}{V_{\delta^{(k \wedge j)}} (x_{\alpha}^k) + V_{\delta^{(k \wedge j)}} (x_{\tau}^j) +V (x_{\alpha}^k,x_{\tau}^j)} \bigg(
\frac{\delta^{(k \wedge j)}}{ \delta^{(k \wedge j)} + d (x_{\alpha}^k,x_{\tau}^j) }
\bigg)^\gamma| \lambda_{\alpha}^k |\chi_{{Q}^{{j}}_{\tau }}(x)
\\
&\leq &C\Big\{\sum_{\alpha \in \mathscr{Y}^k}
        \mu(Q_{\alpha}^k)^{{1\over2}r}\Big[\frac{1}{V_{\delta^{(k \wedge j)}} (x_{\alpha}^k) + V_{\delta^{(k \wedge j)}} (x) +V (x_{\alpha}^k,x)} \bigg(
\frac{\delta^{(k \wedge j)}}{ \delta^{(k \wedge j)} + d (x_{\alpha}^k,x) }
\bigg)^\gamma\Big]^r |\lambda_{\alpha}^k |^r\Big\}^\frac{1}{r}
\\
&\leq&C\Big\{ \int_{X} \sum_{\alpha \in \mathscr{Y}^k}
        \mu(Q_{\alpha}^k)^{r-1}\Big[\frac{1}{V_{\delta^{(k \wedge j)}} (y) + V_{\delta^{(k \wedge j)}} (x) +V (y,x)} \bigg(
\frac{\delta^{(k \wedge j)}}{ \delta^{(k \wedge j)} + d (y,x) }
\bigg)^\gamma\Big]^r\\
&&\times \mu(Q_{\alpha}^k)^{-\frac{r}{2}}|\lambda_{\alpha}^k |^r\chi_{{Q}^{{k}}_{\alpha }}(y)d\mu(y)\Big\}^\frac{1}{r}.
\end{eqnarray*}

Note that $C\delta^{k\omega}\leq  \mu(Q_{\alpha}^k)$ and $r\leq 1$ imply that the last term above is bounded by
\begin{eqnarray*}
&&   C\delta^{k\omega(1 -{1\over r})}\Big\{
\int_{X} \sum_{\alpha \in \mathscr{Y}^k}\Big[\frac{1}{V_{\delta^{(k \wedge j)}} (y) + V_{\delta^{(k \wedge j)}} (x) +V (y,x)} \bigg(
\frac{\delta^{(k \wedge j)}}{ \delta^{(k \wedge j)} + d (y,x) }
\bigg)^\gamma\Big]^r\\
&&\times \mu(Q_{\alpha}^k)^{-\frac{r}{2}}|\lambda_{\alpha}^k |^r\chi_{{Q}^{{k}}_{\alpha }}(y)d\mu(y)\Big\}^\frac{1}{r}
\\&\leq& C\delta^{k\omega(1-\frac{1}{r})}\Big\{  \int_{B}
\sum_{\alpha \in \mathscr{Y}^k}\Big[\frac{1}{V_{\delta^{(k \wedge j)}} (y) + V_{\delta^{(k \wedge j)}} (x) +V (y,x)} \bigg(
\frac{\delta^{(k \wedge j)}}{ \delta^{(k \wedge j)} + d (y,x) }
\bigg)^\gamma\Big]^r\\
&&\times \mu(Q_{\alpha}^k)^{-\frac{r}{2}}|\lambda_{\alpha}^k |^r\chi_{{Q}^{{k}}_{\alpha }}(y)d\mu(y)\Big\}^\frac{1}{r}
\\&&+\delta^{k\omega(1-\frac{1}{r})}\Big\{
\sum_{m=0}^\infty\int_{[\delta^{-(m+1)}B]\setminus[ \delta^{-m}B]} \sum_{\alpha \in \mathscr{Y}^k}\Big[\frac{1}{V_{\delta^{(k \wedge j)}} (y) + V_{\delta^{(k \wedge j)}} (x) +V (y,x)} \\
&&\hskip.5cm\times \bigg(
\frac{\delta^{(k \wedge j)}}{ \delta^{(k \wedge j)} + d (y,x) }
\bigg)^\gamma\Big]^r\mu(Q_{\alpha}^k)^{-\frac{r}{2}}|\lambda_{\alpha}^k |^r\chi_{{Q}^{{k}}_{\alpha }}(y)d\mu(y)\Big\}^\frac{1}{r}
\\
&=:&H_1+H_2.
\end{eqnarray*}

For any $x'\in B,$ we have
\begin{eqnarray*}
H_1&\leq& C\delta^{k\omega(1-\frac{1}{r})}\Big\{\mu(B)^{1- r}\frac{1}{\mu(B) }
\int_{B}\sum_{\alpha \in \mathscr{Y}^k}\mu(Q_{\alpha}^k)^{-\frac{r}{2}}|\lambda_{\alpha}^k |^r\chi_{{Q}^{{k}}_{\alpha }}(y)d\mu(y)\Big\}^\frac{1}{r}
\\
&\leq & C\delta^{k\omega(1-\frac{1}{r})}\mu(B)^{\frac{1}{r}- 1}\Big\{
M\Big(\sum_{\alpha \in \mathscr{Y}^k}\mu(Q_{\alpha}^k)^{-\frac{1}{2}}|\lambda_{\alpha}^k |^r\chi_{{Q}^{{k}}_{\alpha }}\Big)(x')\Big\}^\frac{1}{r}.
\end{eqnarray*}

To estimate the term $H_2,$ applying the doubling property on the measure $\mu$ yields
\begin{eqnarray*}
&&  \int_{[\delta^{-(m+1)}B]\setminus[ \delta^{-m}B]}\sum_{\alpha \in \mathscr{Y}^k}\Big[\frac{1}{V_{\delta^{(k \wedge j)}} (y) + V_{\delta^{(k \wedge j)}} (x) +V (y,x)} \bigg(
\frac{\delta^{(k \wedge j)}}{ \delta^{(k \wedge j)} + d (y,x) }
\bigg)^\gamma\Big]^r\\
&&\times \mu(Q_{\alpha}^k)^{-\frac{r}{2}}|\lambda_{\alpha}^k |^r\chi_{{Q}^{{k}}_{\alpha }}(y)d\mu(y)
\\
&\leq & C\delta^{m\gamma r} \mu(\delta^{-m}B)^{1-r}{1\over \mu(\delta^{-m}B)}\int_{\delta^{-(m+1)}B}\sum_{\alpha \in \mathscr{Y}^k}\mu(Q_{\alpha}^k)^{-\frac{r}{2}}|\lambda_{\alpha}^k |^r\chi_{{Q}^{{k}}_{\alpha }}(y)d\mu(y)
\\
&\leq & C\delta^{m\gamma r}\mu(\delta^{-m}B)^{1- r}M\Big(\sum_{\alpha \in \mathscr{Y}^k}\mu(Q_{\alpha}^k)^{-\frac{r}{2}}|\lambda_{\alpha}^k |^r\chi_{{Q}^{{k}}_{\alpha }}\Big)(x')
\\
&\leq & C\delta^{m[\gamma r-\omega(1- r)]} \mu(B)^{1- r} M\Big(\sum_{\alpha \in \mathscr{Y}^k}\mu(Q_{\alpha}^k)^{-\frac{r}{2}}|\lambda_{\alpha}^k |^r\chi_{{Q}^{{k}}_{\alpha }}\Big)(x').
\end{eqnarray*}
If $\gamma$ is chosen so that $\gamma r-\omega(1-r)>0,$ then
\begin{eqnarray*}
H_2&\leq& C\delta^{k\omega(1-\frac{1}{r})}\mu(B)^{\frac{1}{r}-1}\Big\{
M\Big(\sum_{\alpha \in \mathscr{Y}^k}\mu(Q_{\alpha}^k)^{-\frac{r}{2}}|\lambda_{\alpha}^k |^r\chi_{{Q}^{{k}}_{\alpha }}\Big)(x')\Big\}^\frac{1}{r}.
\end{eqnarray*}
%where $\gamma$ can be arbitrarily large such that $r>\frac{\omega}{\omega+\gamma}$.

This implies that \begin{eqnarray*}
&& \sum \limits_{\tau \in \mathscr{Y}^j}
\sum_{\alpha \in \mathscr{Y}^k}
        \frac{\mu(Q_{\alpha}^k)^{\frac{1}{2}}}{V_{\delta^{(k \wedge j)}} (x_{\alpha}^j) + V_{\delta^{(k \wedge j)}} (x_{\tau}^j) +V (x_{\alpha}^k,x_{\tau}^j)} \bigg(
\frac{\delta^{(k \wedge j)}}{ \delta^{(k \wedge j)} + d (x_{\alpha}^k,x_{\tau}^j) }
\bigg)^\gamma|\lambda_{\alpha}^k |\chi_{{Q}^{{j}}_{\tau }}(x)
\\&\leq& C\delta^{k\omega(1-\frac{1}{r})}\mu(B)^{\frac{1}{r}-1}\Big\{
M\Big(\sum_{\alpha \in \mathscr{Y}^k}\mu(Q_{\alpha}^k)^{-\frac{r}{2}}|\lambda_{\alpha}^k |^r\chi_{{Q}^{{k}}_{\alpha }}\Big)(x')\Big\}^\frac{1}{r}.
\end{eqnarray*}
Taking infimum for $x'$ over $B$ implies the claim \eqref{4.6}.

The estimate \eqref{4.5} together with the estimate \eqref{4.6} yields
\begin{eqnarray*}
&&\sum \limits_{\tau \in \mathscr{Y}^j}
        \mu ( Q_\tau^j )^{-1/2}|\lambda_{\tau}^j |
        \chi_{{Q}^{{j}}_{\tau }}(x)  \nonumber
     \\    &\leq &C
         \sum_{k\in\mathbb{Z}} \delta^{| k - j | \epsilon}\delta^{k\omega(1-\frac{1}{r})}\mu(B)^{\frac{1}{r}-1}\inf_{y\in B}\Big\{
M\Big(\sum_{\alpha \in \mathscr{Y}^k}\mu(Q_{\alpha}^k)^{-\frac{r}{2}}| \lambda_{\alpha}^k |^r\chi_{{Q}^{{k}}_{\alpha }}\Big)(y)\Big\}^\frac{1}{r}
\\    &\leq &C
         \sum_{k\in\mathbb{Z}} \delta^{| k - j | \epsilon}\delta^{k\omega(1-\frac{1}{r})}\mu(B)^{\frac{1}{r}-1}\delta^{ks_2}\inf_{y\in B}\Big\{
M\Big(\sum_{\alpha \in \mathscr{Y}^k}\delta^{-ks_2r}\mu(Q_{\alpha}^k)^{-\frac{r}{2}}|\lambda_{\alpha}^k |^r\chi_{{Q}^{{k}}_{\alpha }}\Big)(y)\Big\}^\frac{1}{r}.\nonumber
\end{eqnarray*}
Choosing $r$ so that $r<\min\{p_2,q_2,1\}$ and denoting
$$
F_k(y)=\Big\{
M\Big(\sum_{\alpha \in \mathscr{Y}^k}\delta^{-ks_2r}\mu(Q_{\alpha}^k)^{-\frac{r}{2}}|\lambda_{\alpha}^k |^r\chi_{{Q}^{{k}}_{\alpha }}\Big)(y)\Big\}^\frac{1}{r},$$
we have
\begin{eqnarray*}
&&  \Big\{\sum_{j=-\infty}^N
                    \sum \limits_{\tau \in \mathscr{Y}^j}\delta^{-js_1q_1}\Big(
        \mu ( Q_\tau^j )^{-1/2}|\lambda_{\tau}^j |
        \chi_{{Q}^{{j}}_{\tau }}(x) \Big)^{q_1}
                        \Big\}^{1/q_1}
     \\    &\leq &C\Big\{\sum_{j=-\infty}^N\Big(
                    \sum_{k\in\mathbb{Z}} \delta^{| k - j | \epsilon}\delta^{k\omega(1-\frac{1}{r})}\mu(B)^{\frac{1}{r}-1}\delta^{-js_1}
                    \delta^{ks_2}\inf_{y\in B}F_k(y) \Big)^{q_1}
                        \Big\}^{1/q_1}.
\end{eqnarray*}
Applying the Fefferman--Stein vector-valued maximal function
inequality \cite{FS} for $r<\min\{p_2,q_2\}$ together with the fact that $\|\{\lambda_\alpha^k\}\|_{\dot{f}^{s_2, q_2}_{p_2}(X)}=1$ yields
\begin{eqnarray*}
\inf_{y\in B}F_k(y) &\leq & \inf_{y\in B}\Big\{
                    \sum_{k\in\mathbb{Z}}(F_{k}(y) )^{q_2}
                        \Big\}^{1/q_2}
                        \\&\leq &C\Big\{\mu(B)^{-1}\int_B
                    \Big(\sum_{k\in\mathbb{Z}}(F_{k}(x) )^{q_2}\Big)^{{p_2}/{q_2}}d\mu(x)
                        \Big\}^{1/p_2}
                         \\&\leq &C\mu(B)^{-1/{p_2}}\Big\|
                    \Big(\sum_{k\in\mathbb{Z}}(F_{k}(y) )^{q_2}\Big)^{{1}/{q_2}}
                        \Big\|_{p_2}
                        \\&\leq &C\mu(B)^{-1/{p_2}}\Big\|
                    \Big\{\sum_{k\in\mathbb{Z}}\Big(\sum_{\alpha \in \mathscr{Y}^{k}}\delta^{-ks_2}\mu(Q_{\alpha}^{k})^{-\frac{1}{2}}|\lambda_{\alpha}^{k} |\chi_{{Q}^{{k}}_{\alpha }}\Big)^{q_2}\Big\}^{{1}/{q_2}}
                        \Big\|_{p_2}\\&\leq &C\mu(B)^{-1/{p_2}}.
\end{eqnarray*}We obtain
 \begin{eqnarray*}
&&  \Big\{\sum_{j=-\infty}^N
                    \sum \limits_{\tau \in \mathscr{Y}^j}\delta^{-js_1q_1}\Big(
        \mu ( Q_\tau^j )^{-1/2}|\lambda_{\tau}^j |
        \chi_{{Q}^{{j}}_{\tau }}(x) \Big)^{q_1}
                        \Big\}^{1/q_1}
     \\    &\leq &C\Big\{\sum_{j=-\infty}^N\Big(
                    \sum_{k\in\mathbb{Z}} \delta^{| k - j | \epsilon}\delta^{k\omega(1-\frac{1}{r})}\mu(B)^{\frac{1}{r}-1}
                    \delta^{-js_1}\delta^{ks_2}\mu(B)^{-1/{p_2}}\Big)^{q_1}
                        \Big\}^{1/q_1}.
\end{eqnarray*}
The crucial point here is that we can choose $r$ such that $\frac{1}{r}-1={1\over p_2}$ and hence $\mu(B)^{\frac{1}{r}-1-{{1}\over{p_2}}}=1.$ Note that $r={{p_2}\over{1+p_2}}<p_2.$ It suffices to show that ${ \dot{f}}^{s_1,q_1}_{p_1}\hookrightarrow
{\dot{f}}^{s_2, q_2}_{p_2}$ for any $q_2>r={{p_2}\over{1+p_2}}$ since ${ \dot{f}}^{s_2,q}_{p_2}\hookrightarrow
{\dot{f}}^{s_2, q_2}_{p_2}$ holds for any $0<q<q_2.$
Under these assumptions with $q_2>r={{p_2}\over{1+p_2}}$, we have
\begin{eqnarray}\label{4.7}
&&  \Big\{\sum_{j=-\infty}^N
                    \sum \limits_{\tau \in \mathscr{Y}^j}\delta^{-js_1q_1}\Big(
        \mu ( Q_\tau^j )^{-1/2}|\lambda_{\tau}^j |
        \chi_{{Q}^{{j}}_{\tau }}(x) \Big)^{q_1}
                        \Big\}^{1/q_1}\nonumber
     \\    &\leq &C\Big\{\sum_{j=-\infty}^N\Big(
                    \sum_{k\in\mathbb{Z}} \delta^{| k - j | \epsilon}\delta^{k\omega(1-\frac{1}{r})}\delta^{-js_1}\delta^{ks_2}\Big)^{q_1}
                        \Big\}^{1/q_1}\nonumber
                        \\    &= &C\Big\{\sum_{j=-\infty}^N\Big(
                    \sum_{k\in\mathbb{Z}} \delta^{| k - j | \epsilon}\delta^{-k\frac{\omega }{p_2}}\delta^{-js_1}\delta^{ks_2}\Big)^{q_1}
                        \Big\}^{1/q_1}\\
                           &\leq &C\Big\{\sum_{j=-\infty}^N \delta^{\frac{-j\omega }{p_1}}\Big(
                    \sum_{k\in\mathbb{Z}}\delta^{|j - k| \epsilon} \delta^{(k-j)(s_1-\frac{\omega }{p_1})}\Big)^{q_1}
                        \Big\}^{1/q_1}\nonumber\\
                           &\leq &C \delta^{\frac{-N\omega }{p_1}},\nonumber
\end{eqnarray}
where we choose $\epsilon$ so that $-\eta<-\epsilon<s_1-\frac{\omega }{p_1}<\epsilon<\eta$. On the other hand, applying H\"older inequality with ${{q_2}\over{q_1}}>1$ and ${{q_2}\over{q_1}}$ inequality with ${{q_2}\over{q_1}}\leq 1$ implies that
\begin{eqnarray}\label{4.8}
&&  \Big\{\sum_{j=N+1}^\infty
                    \sum \limits_{\tau \in \mathscr{Y}^j}\delta^{-js_1q_1}\Big(
        \mu ( Q_\tau^j )^{-1/2}|\lambda_{\tau}^j |
        \chi_{{Q}^{{j}}_{\tau }}(x) \Big)^{q_1}
                        \Big\}^{1/q_1}\nonumber
     \\    &= &\Big\{\sum_{j=N+1}^\infty
\delta^{j(s_2-s_1)q_1}\Big(\sum \limits_{\tau \in \mathscr{Y}^j}\delta^{-js_2}
        \mu ( Q_\tau^j )^{-1/2}|\lambda_{\tau}^j |
        \chi_{{Q}^{{j}}_{\tau }}(x) \Big)^{q_1}
                        \Big\}^{1/q_1}
                        \\    &\leq& C
\delta^{N(\frac{\omega}{p_2}-\frac{\omega}{p_1})}\Big\{\sum_{j=N+1}^\infty
\sum \limits_{\tau \in \mathscr{Y}^j}\delta^{-js_2q_2}\Big(
        \mu ( Q_\tau^j )^{-1/2}|\lambda_{\tau}^j |
        \chi_{{Q}^{{j}}_{\tau }}(x) \Big)^{q_2}
                        \Big\}^{1/q_2},\nonumber
\end{eqnarray}
where we use the fact that $s_2-s_1=\frac{\omega}{p_2}-\frac{\omega}{p_1}>0.$
From \eqref{4.7} and \eqref{4.8}, it follows that
 \begin{eqnarray*}&&\|\lambda_{\tau}^j\|_{{\dot{f}^{s_1, q_1}_{p_1}(X)}}^{p_1}
      \\ &= &p_1\sum_{N=-\infty}^{\infty}\int_{2C\delta^{-\omega N/p_1}2^{1/q_1}}^{2C\delta^{-\omega(N+1)/p_1}2^{1/q_1}} t^{p_1-1}
                        \\&&\times
                        \mu\Big(\Big\{x:\Big\{\sum_{j=-\infty}^\infty\sum \limits_{\tau \in \mathscr{Y}^j}\delta^{-js_1q_1}\Big(
        \mu ( Q_\tau^j )^{-1/2}|\lambda_{\tau}^j |
        \chi_{{Q}^{{j}}_{\tau }}(x) \Big)^{q_1}
                        \Big\}^{1/q_1}>t\Big\}\Big)dt
        \\&\leq &p_1\sum_{N=-\infty}^{\infty}\int_{2C\delta^{-\omega N/p_1}2^{1/q_1}}^{2C\delta^{-\omega(N+1)/p_1}2^{1/q_1}} t^{p_1-1}
                        \\&&\times\mu\Big(\Big\{x:\Big\{\sum_{j=N+1}^\infty\sum \limits_{\tau \in \mathscr{Y}^j}\delta^{-js_1q_1}\Big(
        \mu ( Q_\tau^j )^{-1/2}|\lambda_{\tau}^j |
        \chi_{{Q}^{{j}}_{\tau }}(x) \Big)^{q_1}
                        \Big\}^{1/q_1}>2^{-1/q_1}t/2\Big\}\Big)dt
                        \\&\leq &p_1\sum_{N=-\infty}^{\infty}\int_{2C\delta^{-\omega N/p_1}2^{1/q_1}}^{2C\delta^{-\omega(N+1)/p_1}2^{1/q_1}} t^{p_1-1}
                        \\&&\times\mu\Big(\Big\{x:\Big\{\sum_{j=-\infty}^\infty\sum \limits_{\tau \in \mathscr{Y}^j}\delta^{-js_2q_2}\Big(
        \mu ( Q_\tau^j )^{-1/2}|\lambda_{\tau}^j |
        \chi_{{Q}^{{j}}_{\tau }}(x) \Big)^{q_2}
                        \Big\}^{1/q_2}>C\delta^{N(\frac{\omega}{p_1}-\frac{\omega}{p_2})}2^{-1/q_1}t/2\Big\}\Big)dt
                        \\&\leq &p_1\int_0^\infty  t^{p_1-1}
                        \\&&\times\mu\Big(\Big\{x:\Big\{\sum_{j=-\infty}^\infty\sum \limits_{\tau \in \mathscr{Y}^j}\delta^{-js_2q_2}\Big(
        \mu ( Q_\tau^j )^{-1/2}| \lambda_{\tau}^j |
        \chi_{{Q}^{{j}}_{\tau }}(x) \Big)^{q_2}
                        \Big\}^{1/q_2}>Ct^{p_1/p_2}\Big\}\Big)dt
                        \\&\leq &p_2\int_0^\infty u^{p_2-1}
                        \\&&\times\mu\Big(\Big\{x:\Big\{\sum_{j=-\infty}^\infty\sum \limits_{\tau \in \mathscr{Y}^j}\delta^{-js_2q_2}\Big(
        \mu ( Q_\tau^j )^{-1/2}| \lambda_{\tau}^j |
        \chi_{{Q}^{{j}}_{\tau }}(x) \Big)^{q_2}
                        \Big\}^{1/q_2}>Cu\Big\}\Big)du
                       \\&\leq& C \|\lambda_{\tau}^j\|_{{\dot{f}^{s_2, q_2}_{p_2}(X)}}^{p_2}.
\end{eqnarray*}
The proof of the ``if''-part  of Theorem \ref{embedding} for ${\dot{f}^{s, q}_{p}(X)}$ is concluded.

The similar proof yields that the locally lower bound condition on the measure $\mu$ implies ${ {f}}^{s_1,q_1}_{p_1}\hookrightarrow
{{f}}^{s_2, q_2}_{p_2}$. In fact, we have
\begin{eqnarray}\label{4.71}
\Big\{\sum_{j=0}^N
                    \sum \limits_{\tau \in \mathscr{Y}^j}\delta^{-js_1q_1}\Big(
        \mu ( Q_\tau^j )^{-1/2}|\lambda_{\tau}^j |
        \chi_{{Q}^{{j}}_{\tau }}(x) \Big)^{q_1}
                        \Big\}^{1/q_1}
                           &\leq &C \delta^{\frac{-N\omega }{p_1}}
\end{eqnarray}and for $N\geq -1,$
\begin{eqnarray}\label{4.81}
&&\Big\{\sum_{j=N+1}^\infty
                    \sum \limits_{\tau \in \mathscr{Y}^j}\delta^{-js_1q_1}\Big(
        \mu ( Q_\tau^j )^{-1/2}|\lambda_{\tau}^j |
        \chi_{{Q}^{{j}}_{\tau }}(x) \Big)^{q_1}
                        \Big\}^{1/q_1}\nonumber
        \\ &\leq&
\delta^{N(\frac{\omega}{p_2}-\frac{\omega}{p_1})}\Big\{\sum_{j=N+1}^\infty
\sum \limits_{\tau \in \mathscr{Y}^j}\delta^{-js_2q_2}\Big(
        \mu ( Q_\tau^j )^{-1/2}|\lambda_{\tau}^j |
        \chi_{{Q}^{{j}}_{\tau }}(x) \Big)^{q_2}
                        \Big\}^{1/q_2}.
\end{eqnarray}

We write
 \begin{eqnarray*}&&\|\lambda_{\tau}^j\|_{{f}^{s_1, q_1}_{p_1}(X)}
   \\ &= &p_1\int_0^{2C2^{1/q_1}} t^{p_1-1}\mu\Big\{x:\Big\{\sum_{j=0}^\infty\sum \limits_{\tau \in \mathscr{Y}^j}\delta^{-js_1q_1}\Big(
        \mu ( Q_\tau^j )^{-1/2}|\lambda_{\tau}^j |
        \chi_{{Q}^{{j}}_{\tau }}(x) \Big)^{q_1}
                        \Big\}^{1/q_1}>t\Big\}dt
                        \\ &&+p_1\sum_{N=0}^{\infty}\int_{2C\delta^{-\omega N/p_1}2^{1/q_1}}^{2C\delta^{-\omega(N+1)/p_1}2^{1/q_1}} t^{p_1-1}
                        \\&&\times
                        \mu\Big(\Big\{x:\Big\{\sum_{j=0}^\infty\sum \limits_{\tau \in \mathscr{Y}^j}\delta^{-js_1q_1}\Big(
        \mu ( Q_\tau^j )^{-1/2}|\lambda_{\tau}^j |
        \chi_{{Q}^{{j}}_{\tau }}(x) \Big)^{q_1}
                        \Big\}^{1/q_1}>t\Big\}\Big)dt
                        \\&:=&H+I.
                       \end{eqnarray*}

        We only need to estimate $H$ since the proof for $I$ is the same as the above. By \eqref{4.81} with $N=-1$,
        \begin{eqnarray*}
&&\Big\{\sum_{j=0}^\infty
                    \sum \limits_{\tau \in \mathscr{Y}^j}\delta^{-js_1q_1}\Big(
        \mu ( Q_\tau^j )^{-1/2}|\lambda_{\tau}^j |
        \chi_{{Q}^{{j}}_{\tau }}(x) \Big)^{q_1}
                        \Big\}^{1/q_1}\nonumber
        \\ &\leq&
\delta^{(\frac{\omega}{p_1}-\frac{\omega}{p_2})}\Big\{\sum_{j=0}^\infty
\sum \limits_{\tau \in \mathscr{Y}^j}\delta^{-js_2q_2}\Big(
        \mu ( Q_\tau^j )^{-1/2}|\lambda_{\tau}^j |
        \chi_{{Q}^{{j}}_{\tau }}(x) \Big)^{q_2}
                        \Big\}^{1/q_2}.
\end{eqnarray*}
       Therefore, we obtain
\begin{eqnarray*}H&\lesssim & p_1\int_{0}^{2C2^{1/q_1}}t^{p_1-1}
                                 \mu\Big\{x:\Big\{\sum_{j=0}^\infty\sum \limits_{\tau \in \mathscr{Y}^j}\delta^{-js_2q_2}\Big(
        \mu ( Q_\tau^j )^{-1/2}|\lambda_{\tau}^j |
        \chi_{{Q}^{{j}}_{\tau }}(x) \Big)^{q_2}
                        \Big\}^{1/q_2}>\delta^{(\frac{\omega}{p_2}-\frac{\omega}{p_1})}t\Big\}dt
                        \\&\leq&  \frac{p_1}{p_2}\big(2C\delta^{(\frac{\omega}{p_1}-\frac{\omega}{p_2})}2^{1/q_1}\big)^{p_1-p_2} p_2\int_{0}^{2C\delta^{(\frac{\omega}{p_1}-\frac{\omega}{p_2})}2^{1/q_1}} t^{p_2-1}
                                 \\&&\times\mu\Big\{\sum_{j=0}^\infty\sum \limits_{\tau \in \mathscr{Y}^j}\delta^{-js_2q_2}\Big(
        \mu ( Q_\tau^j )^{-1/2}|\lambda_{\tau}^j |
        \chi_{{Q}^{{j}}_{\tau }}(x) \Big)^{q_2}
                        \Big\}^{1/q_2}>t\Big\}dt\\&                                                                    \\&\leq& C\|\lambda_{\tau}^j\|_{{f^{s_2, q_2}_{p_2}(X)}}^{p_2}.   \end{eqnarray*}

The proof of Theorem \ref{embedding} is complete.
% \end{proof}

\section{Applications }

We would like to point out that Theorem \ref{embedding} shows that geometric conditions on the measure only play a crucial role for the embedding theorem.
As an application of Theorem \ref{embedding}, we provide new embedding theorems of the Besov and Triebel-Lizorkin spaces on spaces of homogeneous type in the sense of Coifman and Weiss. For this purpose, we first recall
test functions and distributions on $(\XX,d,\mu),$ space of homogeneous type in the sense of Coifman and Weiss.
\begin{definition}\label{def-of-test-func-space}
    \textup{(Test functions, \cite{HLW})} Fix $x_0\in X$, $r > 0$, $\gamma
    > 0$ and $\beta\in(0,\eta)$ where $\eta$ is the regularity
    exponent from Theorem \ref{theorem AH orth basi}. A
    function $f$ defined on~$X$ is said to be a {\it test
    function of type $(x_0,r,\beta,\gamma)$ centered at $x_0\in
    X$} if $f$ satisfies the following three conditions.
    \begin{enumerate}
        \item[(i)] \textup{(Size condition)} For all $x\in
            X$,
            \[
                |f(x)|
                \leq C \,\frac{1}{V_{r}(x_0) + V(x,x_0)}
                \Big(\frac{r}{r + d(x,x_0)}\Big)^{\gamma}.
            \]

        \item[(ii)] \textup{(H\"older regularity
            condition)} For all $x$, $y\in X$ with $d(x,y)
            < (2A_0)^{-1}(r + d(x,x_0))$,
            \[
                |f(x) - f(y)|
                \leq C \Big(\frac{d(x,y)}{r + d(x,x_0)}\Big)^{\beta}
                \frac{1}{V_{r}(x_0) + V(x,x_0)} \,
                \Big(\frac{r}{r + d(x,x_0)}\Big)^{\gamma}.
            \]
    \end{enumerate}
\end{definition}

We denote by $G(x_0,r,\beta,\gamma)$ the set of all test
functions of type $(x_0,r,\beta,\gamma)$. The norm of $f$ in $
G(x_0,r,\beta,\gamma)$ is defined by
\[
    \|f\|_{G(x_0,r,\beta,\gamma)}
    := \inf\{C>0:\ {\rm(i)\  and \ (ii)}\ {\rm hold} \}.
\]

For each fixed $x_0$, let $G(\beta,\gamma) :=
G(x_0,1,\beta,\gamma)$. It is easy to check that for each fixed
$x_1\in X$ and $r > 0$, we have $G(x_1,r,\beta,\gamma) =
G(\beta,\gamma)$ with equivalent norms. Furthermore, it is also
easy to see that $G(\beta,\gamma)$ is a Banach space with
respect to the norm on $G(\beta,\gamma)$.

For $0 < \beta < \eta$ and $\gamma > 0$, let
$\GGs(\beta,\gamma)$ be the completion of the space
$G(\eta,\gamma)$ in the norm of $G(\beta,\gamma)$. For $f\in
\GGs(\beta,\gamma)$, define $\|f\|_{\GGs(\beta,\gamma)} :=
\|f\|_{G(\beta,\gamma)}$. Finally, let $G_0(\beta,\gamma)=\{f \in G(\beta,\gamma): \int_X f(x) d\mu(x)=0\}$ and $\GGs_0(\beta,\gamma)=\{f \in \GGs(\beta,\gamma): \int_X f(x) d\mu(x)=0\}.$

\begin{definition}
    \textup{(Distributions)} The \emph{distribution space}
    $(\GGs_0(\beta,\gamma))'$ is defined to be the set of all
    linear functionals $\mathcal{L}$ from $\GGs_0(\beta,\gamma)$
    to $\mathbb{C}$ with the property that there exists $C > 0$
    such that for all $f\in \GGs_0(\beta,\gamma)$,
    \[
        |\mathcal{L}(f)|
        \leq C\|f\|_{\GGs(\beta,\gamma)}.
    \]
    Similarly, $(\GGs(\beta,\gamma))'$ is defined to be the set of all
    linear functionals $\mathcal{L}$ from $\GGs(\beta,\gamma)$
    to $\mathbb{C}.$
\end{definition}

The  Besov and Triebel--Lizorkin spaces on $(\XX, d, \mu)$ are defined as
follows.

\begin{definition}\label{BT1}
Suppose that
    $|s|<\eta$ and $\omega$
    is the upper dimension of~$(X,d,\mu)$. Let $\psi_{\alpha}^{k}$ be a wavelet basis constructed in \cite{AH}. For $\beta\in(0,\eta)$, $\gamma>0$ and $\max\left(\frac{\omega}{\omega +  \eta},\frac{\omega}{\omega + \eta+ s}\right) < p\leq  \infty$ and $0<q\leq \infty,$
    the  Besov space $\dot{B}^{s,q}_{p}(X)$ is the collection of all $f\in (\GGs_0(\beta,\gamma))'$ such that the sequence $\{\langle\psi_\alpha^k, f \rangle\}$ belongs to $\dot{b}^{s, q}_{p}(X)$ and
    $$\|f\|_{\dot{B}^{s, q}_{p}(X)}:=\|\{\langle\psi_\alpha^k, f \rangle\}\|_{\dot{b}^{s, q}_{p}(X)}.$$
    The Besov space ${B}^{s,q}_{p}(X)$ is the collection of all $f\in (\GGs(\beta,\gamma))'$ such that the sequence $\{\langle\psi_\alpha^k, f \rangle\}$ belongs to ${b}^{s, q}_{p}(X)$ and
    $$\|f\|_{{B}^{s, q}_{p}(X)}:=\|\{\langle\psi_\alpha^k, f \rangle\}\|_{{b}^{s, q}_{p}(X)}.$$
    For $\beta\in(0, \eta), \gamma>0$ and $\max\left(\frac{\omega}{\omega +  \eta},\frac{\omega}{\omega + \eta+ s}\right) < p < \infty$, $\max\left(\frac{\omega}{\omega + \eta},\frac{\omega}{\omega + \eta+ s}\right) < q \leq\infty,$ the Triebel--Lizorkin space $\dot{F}^{s,q}_{p}(X)$ is the collection of all $f\in (\GGs_0(\beta,\gamma))'$ such that the sequence $\{\langle\psi_\alpha^k, f \rangle\}$ belongs to $\dot{f}^{s, q}_{p}(X)$ and
    $$\|f\|_{\dot{F}^{s, q}_{p}(X)}:=\|\{\langle\psi_\alpha^k, f \rangle\}\|_{\dot{f}^{s, q}_{p}(X)}.$$
    The Triebel--Lizorkin space ${F}^{s,q}_{p}(X)$ is the collection of all $f\in (\GGs(\beta,\gamma))'$ such that the sequence $\{\langle\psi_\alpha^k, f \rangle\}$ belongs to ${f}^{s, q}_{p}(X)$ and
    $$\|f\|_{{F}^{s, q}_{p}(X)}:=\|\{\langle\psi_\alpha^k, f \rangle\}\|_{{f}^{s, q}_{p}(X)}.$$
    \end{definition}
    We remark that it is routing to verify the above definition is independent of the choice of the wavelet constructed in \cite{AH}. We leave the details to the reader.

We now prove the following
\begin{theorem}\label{embedding1}
(i) Let $\max\big\{\frac{\omega}{\omega+\eta},\frac{\omega}{\omega+\eta+s_i}\big\}<p_i\leq \infty, 0<q\leq \infty,$ $i=1,2$ and $s_1\leq s_2$ with $ -\eta<s_1-\omega/p_1=s_2-\omega/p_2<\eta$. Then
$$
\|f\|_{{\dot { B}}^{s_1,q}_{p_1}}\leq C \|f\|_{{\dot {B}}^{s_2,q}_{p_2}}$$
if and only if the measure $\mu$ has the lower bound and
$$\|f\|_{{{B}}^{s_1,q}_{p_1}}\leq C \|f\|_{{{B}}^{s_2,q}_{p_2}}$$
if and only if the measure $\mu$ has the locally lower bound.

(ii) Let $\max\big\{\frac{\omega}{\omega+\eta},\frac{\omega}{\omega+\eta+s_i}\big\}<p_i< \infty$ and $\max\big\{\frac{\omega}{\omega+\eta},\frac{\omega}{\omega+\eta+s_i}\big\}<q_i\leq \infty$ for $i=1,2,$
and $s_1\leq s_2$ with $ -\eta<s_1-\omega/p_1=s_2-\omega/p_2<\eta$. Then
$$\|f\|_{{\dot {F}}^{s_1,q_1}_{p_1}}\leq C \|f\|_{{\dot {F}}^{s_2,q_2}_{p_2}}$$
if and only if $\mu$ has the lower bound and
$$\|f\|_{{{F}}^{s_1,q_1}_{p_1}}\leq C \|f\|_{{{F}}^{s_2,q_2}_{p_2}}$$
if and only if $\mu$ has the locally lower bound.
\end{theorem}
The proof will follow from Theorem \ref{embedding}. More precisely, taking $f(x)=\psi_{\alpha_0}^{k_0}$ gives
\begin{eqnarray*}
\|f\|_{\dot{{B}}^{s_1, q}_{p_1}(X)} &=& \delta^{-k_0s_1} \mu ( Q_{\alpha_0}^{k_0} )^{1/{p_1}-1/2}.
\end{eqnarray*}
Similarly,
\begin{eqnarray*}
\|f\|_{\dot{{B}}^{s_2, q}_{p_2}(X)} & =& \delta^{-k_0s_2} \mu ( Q_{\alpha_0}^{k_0} )^{1/{p_2}-1/2}.
\end{eqnarray*}
Therefore, if
\begin{eqnarray*}
\|f\|_{\dot{{B}}^{s_1, q}_{p_1}}\leq C \|f\|_{\dot{{B}}^{s_2, q}_{p_2}},
\end{eqnarray*}
we should have $\delta^{-k_0s_2} \mu ( Q_{\alpha_0}^{k_0} )^{1/{p_2}-1/2}\geq C \delta^{-k_0s_1} \mu ( Q_{\alpha_0}^{k_0} )^{1/{p_1}-1/2}$ and this implies that $\mu( Q_{\alpha_0}^{k_0} )\geq C\delta^{k_0\omega}$, for any $ k_0\in\mathbb{Z}$ and $\alpha_0\in\mathscr{Y}^{k_0}$. Repeating the same proof implies that if
\begin{eqnarray*}
\|f\|_{{B}^{s_1, q}_{p_1}}\leq C \|f\|_{ {B}^{s_2, q}_{p_2}},
\end{eqnarray*}
then $\mu( Q_{\alpha_0}^{k_0} )\geq C\delta^{k_0\omega}$, for any $ k_0\in\mathbb{Z}^+,\ \ \alpha_0\in\mathscr{Y}^{k_0}$.

The `` only if" parts for $\dot{F}^{s, q}_{p}(X)$ and ${F}^{s, q}_{p}(X)$ can be verified similarly.

To show the the ``if" parts, by the definition of $\dot{{B}}^{s, q}_{p}$ and Theorem \ref{embedding},
$$\|f\|_{\dot{{B}}^{s_1, q}_{p_1}}=\|\{<f,\psi_\alpha^k>\}\|_{\dot{{b}}^{s_1, q}_{p_1}}\leq C \|\{<f,\psi_\alpha^k>\}\|_{\dot{{b}}^{s_2, q}_{p_2}}
=C\|f\|_{\dot{{B}}^{s_1, q}_{p_1}}.$$
The proofs for ${B}^{s, q}_{p}(X), \dot{F}^{s, q}_{p}(X)$ and ${F}^{s, q}_{p}(X)$ are same and we omt the details of the proofs.

The second application of Theorem \ref{embedding} is to give the necessary and sufficient conditions for the classical weighted Besov and Triebel-Lizorkin spaces, in particular, the weighted Sobolev spaces on $\mathbb R^n.$ We now recall these spaces.

Let $\psi$ be a non-negative function in $\mathscr{S}$ on $\mathbb R^n$ such that supp$\psi=\{ 1/2\leq |x|\leq2 \}$, $\psi(x)\geq 0$ for $1/2\leq |x|\leq2$ and
$\sum_{j=-\infty}^\infty \psi(2^{-j}x)=1$ for $|x|\not=0$.
Let $\psi_j$, $j=0,\pm1,\pm2,\ldots,$ and $\Psi$ be functions in $\mathcal{S}$ given by
$$ \hat{\psi}_j(x)= \psi(2^{-j}x),\quad \hat{\Psi}(x)=1-\sum_{j=1}^\infty \hat{\psi}_j(x).  $$
Weighted Besov and Triebel spaces are defined as follows. For $-\infty<s<\infty$ and $0<p,q\leq \infty,$
\begin{eqnarray*}
B^{s,q}_{p,w}(\mathbb R^n)=\Big\{f\in \mathcal{S}':  \|f\|_{B^{s,q}_{p,w}}= \|\Psi*f\|_{p,w}+  \{ \sum_{j=1}^\infty (2^{js}\|\psi_j*f\|_{p,w})^q\}^{1/q}<\infty \Big\}
\end{eqnarray*}
and
\begin{eqnarray*}
\dot{B}^{s,q}_{p,w}(\mathbb R^n)=\Big\{f\in \mathcal{S_\infty}':  \|f\|_{\dot{B}^{s,q}_{p,w}}=\{ \sum_{j=-\infty}^\infty (2^{js}\|\psi_j*f\|_{p,w})^q\}^{1/q}<\infty \Big\}.
\end{eqnarray*}
For $-\infty<s<\infty$, $0<p< \infty$ and $0<q\leq \infty,$
\begin{eqnarray*}
F^{s,q}_{p,w}(\mathbb R^n)=\Big\{f\in \mathcal{S}':  \|f\|_{F^{s,q}_{p,w}}=\|\Psi*f\|_{p,w}+  \|\{ \sum_{j=1}^\infty (2^{js}|\psi_j*f|)^q\}^{1/q}\|_{p,w}<\infty \Big\}\\
\end{eqnarray*}
and
\begin{eqnarray*}
\dot{F}^{s,q}_{p,w}(\mathbb R^n)=\Big\{f\in \mathcal{S_\infty}':  \|f\|_{\dot{F}^{s,q}_{p,w}}= \|\{ \sum_{j=-\infty}^\infty (2^{js}|\psi_j*f|)^q\}^{1/q}\|_{p,w}<\infty \Big\}.
\end{eqnarray*}
We would like to point out that ${F}^{s,2}_{p,w}(\mathbb R^n)$ and $\dot{F}^{s,2}_{p,w}(\mathbb R^n)$ are the weighted Sobolev spaces.

Suppose that weights $w$ satisfy Muckenhoupt $A_p$ condition, $1<p\leq\infty.$ The embedding theorem for weighted Besov and Triebel spaces with $A_p$ weights is the following
\begin{theorem}
(i) If $-\infty<s_1\leq s_0<\infty, 0<p_0\leq p_1\leq \infty$ and $s_0-n/p_0=s_1-n/p_1$, then
$ B^{s_0,q}_{p_0,w} \hookrightarrow B^{s_1,q}_{p_1,w}$ if and only if $w(B(x,r))\geq cr^n$ for all $x$ and $0<r\leq1,$ and $ \dot{B}^{s_0,q}_{p_0,w} \hookrightarrow \dot{B}^{s_1,q}_{p_1,w}$ then if and only if $w(B(x,r))\geq cr^d$ for all $x$ and $0<r<\infty$.

(ii) If $-\infty<s_1< s_0<\infty, 0<p_0< p_1\leq \infty, 0<q_1, q_2\leq\infty$ and $s_0-n/p_0=s_1-n/p_1$, then
$ F^{s_0,q_1}_{p_0,w}\hookrightarrow F^{s_1,q_2}_{p_1,w} $ if and only if $w(B(x,r))\geq cr^d$ for all $x$ and $0<r\leq1,$ and $ \dot{F}^{s_0,q_1}_{p_0,w} \hookrightarrow \dot{F}^{s_1,q_2}_{p_1,w} $ if and only if $w(B(x,r))\geq cr^d$ for all $x$ and $0<r\infty$.
\end{theorem}
The ``if" parts of this theorem were proved in \cite{B}. To show the ``only if" parts, it suffices to consider $1<p,q<\infty$ and $-1<s<1.$ To this end, we first recall the wavelet basis on $R^n$ given in \cite{M}.
\begin{theorem}
There exist $2^n - 1$ functions $\psi_1, \cdot\cdot\cdot, \psi_q$ having the following two properties:
$$|\partial^\alpha\psi_i(x)|\leq C_N(1+|x|)^{-N}$$
for every multi-index $\alpha\in \mathbb N^n$ such that $|\alpha|\leq r,$ each $x\in \mathbb R^n$ and every $N\geq 1;$
$$\int x^\alpha \psi_i(x) dx=0,$$
for $|\alpha|\leq r$ and $1\leq i\leq 2^n-1.$ Moreover, the functions $2^{nj/2}\psi_i(2^jx-k), 1\leq i\leq q, k\in \mathbb Z^n, j\in \mathbb Z,$ form an orthonormal basis of $L^2(R^n).$
\end{theorem}
For each $j\in \mathbb Z,$ and $k\in \mathbb Z^n,$ let $Q(j,k)$ denote the dyadic cube defined by $2^jx-k\in [0, 1)^n.$ Now we introduce the sequence spaces as follows.
\begin{definition}\label{BT2}
We say that a sequence $\{\lambda_k^j\}$ belongs to $\dot{b}^{s,q}_{p,w}$ if
    $$\|\{\lambda_k^j\}\|_{\dot{{b}}^{s, q}_{p,w}} :=\Big\{\sum_{j\in\mathbb{Z}}\delta^{-ksq}
                        \bigg[ \sum_{ k\in \mathbb Z^n} \Big( w( Q(j,k))^{1/p-1/2} |\lambda_k^j| \Big)^p \bigg]^{q/p}\Big\}^{1/q}<\infty$$
    and a sequence $\{\lambda_k^j\}_{j\in \mathbb{ Z}^+,{k\in\mathbb{Z^n}}}$ belongs to ${b}^{s,q}_{p,w}$ if
    $$\|\{\lambda_k^j\}\|_{{{ b}}^{s, q}_{p,w}} :=\Big\{\sum_{j\in\mathbb{Z}^+}\delta^{-ksq}
                        \bigg[ \sum_{ k\in \mathbb Z^n} \Big( w( Q(j,k ))^{1/p-1/2} |\lambda_k^j| \Big)^p \bigg]^{q/p}\Big\}^{1/q}<\infty,$$
where $Q(j,k)$ are dyadic cubes in $\mathbb R^n.$

A sequence $\{\lambda_k^j\}_{j\in \mathbb Z, k\in\mathbb{Z^n}}$ belongs to $\dot{f}^{s,q}_{p,w}$ if
             $$\|\{\lambda_k^j\}\|_{{\dot {f}}^{s, q}_{p,w}} :=\Big\|\Big\{\sum_{j\in\mathbb{Z}}
                    \sum \limits_{k\in \mathbb Z^n}\delta^{-ksq}\Big(
        w( (Q(j,k))^{-1/2}|\lambda_k^j |
        \chi_{Q(j,k)}(x)\Big)^q\Big\}^{1/q}\Big\|_{L^p(X)}<\infty, $$
and a sequence $\{\lambda_k^j\}_{j\in \mathbb Z^+, k\in\mathbb{Z^n}}$ belongs to ${f}^{s,q}_{p,w}$ if
$$\|\{\lambda_k^j\}\|_{{{f}}^{s, q}_{p}} :=\Big\|\Big\{\sum_{j\in\mathbb{Z}^+}
                    \sum \limits_{k\in \mathbb Z^n}\delta^{-ksq}\Big(
        w( Q(j,k))^{-1/2}|\lambda_k^j |
        \chi_{Q(j,k)}(x)\Big)^q\Big\}^{1/q}\Big\|_{L^p(X)}<\infty.$$
\end{definition}
It is a routing business to verify that $f\in B^{s,q}_{p,w}$ if and only if $\{\langle f, 2^{nj/2}\psi_i(2^j\cdot-k)\rangle\}\in b_{s,q}^{p,w}, 1\leq i\leq 2^n-1.$ Therefore, if $ B_{s_0,q}^{p_0,w}\hookrightarrow B_{s_1,q}^{p_1,w}$ then $ b_{s_0,q}^{p_0,w}\hookrightarrow  b_{s_1,q}^{p_1,w}.$ By Theorem \ref{embedding}, $w(B(x,r))\geq cr^n$ for all $x\in \mathbb R^n$ and $0<r\leq 1.$ The other proofs are similar and we omit the details.

We would like to point out that, in general, the $A_p$ condition can not imply the lower bound property. To see this, let
$w_{\alpha,\beta}(x)$ equal $|x|^\alpha$ if $|x|\leq 1$ and $|x|^\beta$ if $|x|>1.$
It is easy to check that if $ -n<\beta<\alpha < n(p-1),$ then $w_{\alpha,\beta}\in A_p.$ However, the following inequality
$$ w(B(x,r))\geq Cr^n $$
can not hold for any fixed constant $C.$

\bigskip

\noindent {\bf Acknowledgement: } The third author would like to thank Professor Thierry Coulhon for drawing author's attention to \cite{CKP} and \cite{SC} and Professor Qui Bui for interesting the lower bound property for the weighted Besov and Triebel-Lizorkin spaces on $\mathbb R^n.$

%------------------------------------------------------------------

\bigskip
%\bigskip

\noindent School  of Mathematic Sciences, South China Normal University, Guangzhou, 510631, P.R. China.

\noindent {\it E-mail address}:
\texttt{20051017@m.scnu.edu.cn}

\medskip
%\vskip 0.5cm

\noindent Department of Mathematics, Auburn University, AL
36849-5310, USA.

\noindent {\it E-mail address}: \texttt{hanyong@auburn.edu}

\medskip
%\vskip 0.5cm

\noindent Department of Mathematics, Macquarie University, NSW, 2109, Australia.

\noindent {\it E-mail address}: \texttt{ji.li@mq.edu.au}

\end{document}